\documentclass[a4paper, 12pt]{article}

\usepackage[top=1cm, bottom=2cm, left=1cm, right=1cm]{geometry}
\usepackage[parfill]{parskip}    		
\usepackage{hyperref}
\hypersetup{hidelinks}

\usepackage{fontspec}
\usepackage{polyglossia}
\setmainlanguage{english}
\setotherlanguage{russian}

\newfontfamily{\russianfont}{cmunrm.otf}[
    BoldFont = cmunbx.otf,
    ItalicFont = cmunti.otf,
    BoldItalicFont = cmunbi.otf,
    Script = Cyrillic
]

\usepackage{mathtools, amssymb, amsthm}
\allowdisplaybreaks
\usepackage{relsize}
\usepackage{arcs}
\usepackage[mathscr]{eucal}
\usepackage{enumitem}
\setenumerate{
	rightmargin=5pt,      
	labelsep=5pt        
	}
\usepackage{multicol}

\usepackage[
	backend = biber,
	style = numeric,
	bibencoding = utf8,
	giveninits = true,
	url = false,
	doi = false,
	isbn = false,
	eprint = false,
	autolang = langname
	]{biblatex} 
\AtEveryBibitem{\clearlist{language}}
\DefineBibliographyStrings{russian}{in = {//}}

\newtheoremstyle{custom}
	{0pt}
	{0pt}
	{\itshape}
	{\parindent}
	{\bfseries}
	{}
	{\newline}
	{}

\theoremstyle{custom}
\newtheorem{theorem}{Theorem}[subsection]
\newtheorem{lemma}[theorem]{Lemma}
\newtheorem{definition}[theorem]{Definition}
\newtheorem{proposition}[theorem]{Proposition}
\newtheorem{corollary}[theorem]{Corollary}

\renewcommand {\qedsymbol}{$\blacksquare$}
\renewenvironment{proof}{\par\noindent\textbf{Proof:}\\}{\hfill\qedsymbol}

\newenvironment{note}
{%
  \begin{center}%
    {\small\bfseries Note}%
    \vspace{-0.5em}%
  \end{center}%
  \small%
  \quotation
}
{\endquotation}

\newenvironment{keywords}
{%
  \begin{center}%
    {\small\bfseries Keywords:}%
    \vspace{-0.5em}%
  \end{center}%
  \small%
  \quotation
}
{\endquotation}

\newcommand{\Top}{\mathscr{T}}
\newcommand{\diff}{\mathop{}\!d}

\DeclareMathOperator{\supp}{supp}
\DeclareMathOperator{\im}{im}
\DeclareMathOperator*{\esssup}{ess\,sup}
\DeclareMathOperator{\Lift}{\mathbb{L}}
\renewcommand{\restriction}{\mathord{\upharpoonright}}
\newcommand{\CA}{\mathrm{CA}}
\newcommand{\SM}{\mathcal{SM}}

\newcommand{\Sim}{\mathcal{S}}
\newcommand{\V}{\mathbb{V}}
\newcommand{\W}{\mathbb{W}}
\newcommand{\F}{\mathbb{F}}
\newcommand{\A}{\mathcal{A}}
\newcommand{\var}{\mathrm{v}}
\newcommand{\svar}{\mathrm{sv}}
\newcommand{\M}{\mathbb{M}}
\newcommand{\BSM}{\mathcal{BSM}}
\newcommand{\TF}{\mathfrak{F}}
\newcommand{\TM}{\mathfrak{M}}
\newcommand{\pairing}[2]{\left\langle #1 \:\vphantom{\big\vert}\middle\vert\: #2 \right\rangle}
\newcommand{\DS}{D\!S}

\title{Towards a Theory of Dobrakov-Sobolev Spaces}
\author{Artem Yurievich Dudko}

\begin{document}

\begin{titlepage}

\centering

\vspace*{2cm}

{\huge \bfseries Towards a Theory of Dobrakov-Sobolev Spaces \par}

\vspace{0.5cm}

{\Large Artem Yurievich Dudko \par}

\vspace{0.3cm}

{\large \texttt{artemdudko@gmail.com}  \qquad \texttt{artem.dudko@math.msu.ru} \par}

\vspace{1.5cm}

\begin{abstract}
	\noindent
	The aim of this paper is to introduce a generalization of Sobolev spaces based on the Dobrakov integral.
	More precisely, we consider the setting of Banach-valued functions and Fomin differentiable Borel operator-valued measures on a finite-dimensional space.
	To build the necessary rigorous foundation, we establish analogs of several key results from the theory of differentiable real-valued measures, including the Leibniz rule and the integration by parts formula, all within the context of Dobrakov integration.
	These results are then embedded into the general scheme of vector-valued distribution theory.
	In particular, we describe the configuration of test spaces that yields an appropriate definition of a generalized derivative with respect to a differentiable operator-valued measure.
\end{abstract}


\vfill

{\large 25 June 2026}

\end{titlepage}

\setcounter{page}{2}
\tableofcontents

\newpage
\section{Introduction}

In the standard definition of Sobolev spaces of scalar-valued functions on $\mathbb{R}^d$ a crucial role is played by the notion of generalized derivative, the definition of which in this context significantly depends upon the properties of the Lebesgue measure and integration by parts formula.
The construction can be extended to accommodate Banach-valued functions using the Bochner integral in place of the Lebesgue integral.
Another possible direction for generalization is to base the construction on a different measure and naturally along with it to consider a more general domain such as a locally convex space.

Several well-known approaches to implementing this idea in case of a real-valued measure on a locally convex space are reviewed in \autocite{Bogachev_2023}.
One of these approaches is analogous to the classical definition via generalized derivatives and relies on the theory of differentiable measures.
In particular, it supposes that the given real-valued measure is differentiable in the sense of Fomin and has its logarithmic derivatives in $L_q$, where $p^{-1}+q^{-1}=1$.
If we take $\mathbb{R}^d$ as the locally convex space in question, then for real-valued functions the corresponding definition of the Sobolev space of order $(1,p)$ given in \autocite{Bogachev_2023} reduces to:

\begin{definition} \label{0}
	The Sobolev space $G_p^1(\mu)$ consists of all functions $f \in L_p(\mu)$, having Sobolev derivatives $\partial_hf \in L_p(\mu)$ for all $h \in \mathbb{R}^d$ in the sense of integration by parts formula:
	$$ \forall \varphi \in C_B^{\infty}(\mathbb{R}^d, \mathbb{R}):      \;       \int_{\mathbb{R}^d} \partial_hf(x) \varphi(x) \mu(\diff{x})    =   - \int_{\mathbb{R}^d} f(x) \partial_h\varphi(x) \mu(\diff{x})   -      \int_{\mathbb{R}^d} f(x) \varphi(x) \Theta_h(x) \mu(\diff{x}) $$
	where $\Theta_h \in L_q(\mu)$ is the logarithmic derivative of $\mu$ along $h$, and $C_B^{\infty}(\mathbb{R}^d, \mathbb{R})$ is the space of smooth functions with bounded derivatives of all orders.
\end{definition}

The purpose of this work is to give a rigorously justified definition of a more general version of the space above, while for the time being still confining ourselves to the domain of  $\mathbb{R}^d$.
More specifically, given Banach spaces $\V$ and $\W$, we will consider functions on $\mathbb{R}^d$ with values in $\V$ and a Borel measure on $\mathbb{R}^d$ with values in the space $L(\V,\W)$ of bounded linear operators, countably additive in the operator norm.
Moreover, we will use the Dobrakov integral as the main analytical backbone, which will allow to avoid the assumption of finite variation.
Fulfilling this goal will require a considerable amount of preparatory work.
Firstly, the existing body of literature on the Dobrakov integral is rather limited, and we will need to obtain a few technical results that are not explicitly found elsewhere.
Secondly, when it comes to the theory of differentiable measures, the existing literature mainly focuses on the real-valued case, and very little is available in the vector-valued setting beyond the basic definitions.
Therefore, some fundamental results for real-valued measures leading up to the Definition \ref{0} must be reestablished within the vector-valued setting and the Dobrakov integration framework.
Lastly, it seems desirable to carefully embed the construction into a more general framework of vector-valued distribution theory, showing precisely how the notion of Sobolev derivative, described in the Definition \ref{0}, is linked to that of distributional and classical derivatives.

The plan for this paper is as follows.

Section \ref{1} is devoted to Dobrakov integration.
For the sake of self-containment we will review some of the basic concepts from vector measure theory, that will be used throughout.
Also we will briefly describe the construction of the Dobrakov integral together with some of the important theorems forming the basis of this theory.
All results already available in the literature will be stated without proof, giving the appropriate references.
Subsection \ref{29} contains some new technical results concerning integration with respect to a measure, defined by an operator-valued density, and the corresponding change of measure formula.
These will be necessary to handle the operator-valued logarithmic derivatives later on.

Section \ref{34} is dedicated to the concepts of continuity and differentiability of vector- and operator-valued measures on $\mathbb{R}^d$.
Here we obtain analogs of some well-known classical results for real-valued measures, culminating in Lebniz rule and integration by parts formula in connection with the Dobrakov integral, all necessary for defining the generalized derivative in the succeeding section.

Section \ref{57} is concerned with vector-valued distributions on $\mathbb{R}^d$ with respect to a differentiable Borel operator-valued measure.
Here we provide a suitable configuration of test spaces leading to the desirable definition of a generalized derivative that matches Definition \ref{0}.

Finally, section \ref{58} defines the Dobrakov-Sobolev space of order $(k,p)$ and proves its completeness.
  
\newpage
\section{Dobrakov integral} \label{1}

Dobrakov developed his theory of integration over a period of 18 years starting in 1970, which resulted in a series of articles.
Among these papers \autocite{Dobrakov_0, Dobrakov_1, Dobrakov_2} play a fundamental role, containing the basic notions, theorems on interchange of limit and integral, and various results pertaining to Lebesgue type spaces that arise within this theory.
Special mention must be made of article \autocite{Dobrakov_Panchapagesan} that came out somewhat later and complements work \autocite{Dobrakov_1}, filling in some logical gaps and containing an exhaustive treatment of the Pettis Measurability Theorem.
An excellent overview of Dobrakov's theory and its comparison with other Lebesgue-type integrals can be found in \autocite{Panchapagesan_ODF}.

As discussed in \autocite{Panchapagesan_ODF}, the Dobrakov integral coincides with the classical Lebesgue integral in the scalar case and subsumes the Bochner integral as well as its generalization to the case of operator-valued measures with finite variation, known as the Dinculeanu integral \autocite{Dinculeanu_VM}.
One of its distinguishing features is absence of absolute integrability, despite which it still possesses an analog of the $L_p$-norm, that in some sense constitutes an extension of the notion of semivariation.
In its original formulation the framework allows integration of Banach-valued functions with respect to an operator-valued measure, countably additive in the strong operator topology, defined on a $\delta$-ring and having finite semivariation.
However this level of generality will not be attained here due to some of our results in the next section requiring the structure of a $\sigma$-algebra and continuity of semivariation (the later immidiately elevates countable additivity to the uniform operator topology).
For these reasons when presenting the basics of integration theory we will confine ourselves to this less general setting.

\subsection{Vector measures and the associated submeasures} \label{2}
For the rest of section \ref{1} we fix a measurable space with a $\sigma$-algebra $(X, \A)$.
By $\mathrm{FD}(\A)$ we will denote the collection of all finite disjoint families of sets in $\A$.
The classes of Borel and compact subsets of a topological space $Y$ are denoted by $\mathcal{B}_Y$ and $\mathcal{K}_Y$ respectively.
$Y$, $\V$, $\W$ are Banach spaces over $\F \in \{\mathbb{R}, \mathbb{C}\}$.

\begin{definition}
    	An additive function $\mu : \A \to Y$ is called a vector measure.
    	The set of all $Y$-valued measures on $\A$, countably additive in the norm topology, will be denoted by $\CA\left[\A,Y\right]$.
	A family $\mathcal{M} \subseteq \CA\left[\A,Y\right]$ is called uniformly countably additive, if for any disjoint sequence $\{A_n\}$ the series $\sum_{i=1}^{\infty} \mu(A_i)$ converges uniformly in $\mu \in \mathcal{M}$.
\end{definition}

The next fundamental theorem plays a vital role in the construction of Dobrakov integral.
\begin{theorem}[Vitali-Hahn-Saks-Nikodym, \autocite{Dunford_Schwartz_LO1}, Theorem IV.10.5] \label{3}
	Suppose $\{\mu_n\}_1^\infty$ is a sequence in $\CA\left[\A,Y\right]$, and for each $A \in \A$ there exists a limit $\mu(A) = \lim_{n\to\infty} \mu_n(A)$.\\
	Then: $\mu \in \CA\left[\A,Y\right]$ and the family $\{\mu_n\}_1^\infty$ is uniformly countably additive on $\A$.
\end{theorem}

\begin{definition}
	A nondecreasing function $\lambda : \A \to [0, +\infty]$ with $\lambda(\varnothing)=0$ is called a submeasure.\\
	A submeasure is called continuous if for any monotone sequence $\A \ni A_n \searrow \varnothing$ we have $\lambda(A_n) \rightarrow 0$.
\end{definition}

Let us recall the important notions of variation, semivariation and supremation, arising naturally in the theory of vector measures.

\begin{definition}[\autocite{Dinculeanu_VM} p. 32]
	Variation of a vector measure $\mu : \A \to Y$ on a subset $E \subseteq X$ is defined as:
	$$\var(\mu, E) := \sup\left\{ \sum_{i=1}^{m}\|\mu(A_i)\|_Y \;:\; \{A_i\}_1^m \in \mathrm{FD}(\A),\; \bigsqcup_{i=1}^{m}A_i \subseteq E \right\}$$
	Its restriction to $\A$ will be denoted by $\overline{\mu}$.
\end{definition}

\begin{definition}[\autocite{Dinculeanu_VM}, p. 51]
	Suppose there is a linear isometry $Y \hookrightarrow L(\V, \W)$.
	Semivariation of a vector measure $\mu : \A \to Y$ relatively to $(\V, \W)$ on a subset $E \subseteq X$ is defined as:
	$$\svar_{\V, \W}(\mu, E) := \sup\left\{  \left\| \sum_{i=1}^{m} \mu(A_i) v_i \right\|_{\W} \;:\; \{A_i\}_1^m \in \mathrm{FD}(\A),\; \bigsqcup_{i=1}^{m}A_i \subseteq E, \; \|v_i\|_\V \leq 1 \right\}$$
	Its restriction to $\A$ will be denoted by $\widehat{\mu}$.
\end{definition}

According to a corollary in \autocite[p. 55]{Dinculeanu_VM}, for $\mu \in \CA\left[\A,Y\right]$ the semivariation $\widehat{\mu}$ is a countably subadditive sumbeasure on $\A$ and,
as mentioned in \autocite[p. 17]{Dobrakov_0}, its continuity on $\A$ implies its finiteness.

Let us note that there is always a natural isometry $Y \hookrightarrow L(\F, Y)$.
Semivariation of a vector measure $\mu : \A \to Y$ relatively to this isometry is called scalar.
Its restriction to $\A$ will always be denoted by $\widetilde{\mu}$.
Thus for $\mu \in \CA\big[\A, L(\V, \W) \big]$ we have both semivariations $\widetilde{\mu}$ and $\widehat{\mu}$.

\begin{proposition}[\autocite{Dinculeanu_VM}, Proposition 4.4] \label{4}
	Suppose $\mu : \A \to Y$ is a vector measure and $Y \hookrightarrow L(\V, \F) = \V^*$.
	Then $\svar_{\V, \F}(\mu, \cdot) = \var(\mu,  \cdot)$
\end{proposition}

Now let $\W'$ be another Banach space and $T \in L(\W, \W')$.
If $\mu : \A \to Y \hookrightarrow L(\V, \W)$, we can define the composition of the measure and the linear operator as follows: 
$$T\mu \;:\; \A \ni A \:\mapsto\: T \circ (\mu(A)) \:\in\: L(\V, \W')$$
It's easily verified that $T\mu$ inherits countable additivity in the operator norm and $\svar_{\V, \W'}(T\mu, \cdot) \leq \|T\|_{op} \cdot \svar_{\V, \W}(\mu, \cdot)$

In particular, given a bounded functional $w^* \in \W^*$ and a vector measure $\mu : \A \to Y \hookrightarrow L(\V, \W)$,  we have their composition $w^*\mu$, which is used in the following important result, expressing semivariation in terms of variations of a family of measures.

\begin{theorem}[\autocite{Dinculeanu_VM}, Proposition 4.5]  \label{5}
	Suppose $\mu : \A \to Y \hookrightarrow L(\V, \W)$ is a vector measure and $\mathcal{N} \subseteq \W^*_1$ is a norming subset.
	Then $$\svar_{\V, \W}(\mu, E) = \sup \big\{ \var(w^* \mu, E) \;:\; w^* \in \mathcal{N} \big\} = \sup \big\{ \svar_{\V, \F}(w^* \mu, E) \;:\; w^* \in \mathcal{N} \big\}$$
\end{theorem}

\begin{definition}[\autocite{Dobrakov_Panchapagesan}, Definition 4]
	Supremation of a vector measure $\mu : \A \to Y$ on a subset $E \subseteq X$ is defined as:
	$$\mathrm{sp}(\mu, E) := \sup\left\{ \|\mu(A)\|_Y \;:\; \A \ni A \subseteq E\right\} $$
	Its restriction to $\A$ will be denoted by $\overset{\smile}{\mu}$.
\end{definition}

\begin{proposition}[\autocite{Dobrakov_Panchapagesan}, Propositions 2,4] \label{36}
	Suppose $\mu \in \CA\left[\A,Y\right]$ is a countably additive measure. Then
	\begin{enumerate}[label=\alph*)]		
		\item
		$ \tilde{\mu}, \overset{\smile}{\mu}$ are continuous submeasures on $\A$
		\item
		$\forall A \in \A \::\; \overset{\smile}{\mu}(A) \leq \tilde{\mu}(A) \leq 4 \overset{\smile}{\mu}(A) < \infty$
	\end{enumerate}
\end{proposition}

In particular, it follows from this proposition that the scalar semivariation of a contably additive operator-valued measure $\mu: \A \to L(\V, \W)$ is continuous and finite.
However, in general countable additivity in the uniform operator topology does not imply continuity or even finiteness of $\widehat{\mu}$.

\begin{theorem}[\autocite{Dobrakov_0}, Lemma 3] \label{6}
	Suppose $\{\mu_n\}_1^\infty$ is a sequence in $\CA\left[\A,L(\V,\W)\right]$, and for each $A \in \A$ there exists a limit $\mu(A) = \lim_{n\to\infty} \mu_n(A)$.
	Then $\mu \in \CA\left[\A,L(\V,\W)\right]$ and if $\widehat{\mu_n}$ are continuous, then so is $\widehat{\mu}$.
\end{theorem}

\begin{definition}
	A vector measure $\mu : \A \to Y$ is said to be continuous with respect to a submeasure $\lambda : \A \to [0,\infty]$ (denoted $\mu << \lambda$), if $\forall \epsilon > 0 :\: \exists \delta > 0 :\: \forall A\in\A :\: \lambda(A)<\delta \implies \|\mu(A)\| < \epsilon$.
\end{definition}

An analogous definition applies when $\mu$ is a submeasure.
Moreover, from the previous proposition we obviously have:
$$\mu << \lambda \iff \tilde{\mu} << \lambda \iff \overset{\smile}{\mu} << \lambda$$

\begin{theorem}[\autocite{Dobrakov_Panchapagesan}, Theorem 6] \label{7}
	Suppose $\mu \in \CA\left[\A,Y\right]$ is a countably additive vector measure and $\lambda : \A \to [0,\infty]$ is a submeasure.\\
	Then the following conditions are equivalent:
	\begin{enumerate}[label=\alph*)]		
		\item
		$\forall \epsilon > 0 :\: \exists \delta > 0 :\: \forall A\in\A :\: \lambda(A)<\delta \implies \|\mu(A)\| < \epsilon$
		
		\item
		$\forall B \in \A: \: \lambda(A)=0 \implies \mu(A)=0$
	\end{enumerate}
\end{theorem}

Let us also recall the definition of a regular measure, when $(X, \Top_X)$ is a locally compact topological space.
\begin{definition}[\autocite{Dinculeanu_Kluvanek_OVM}, p. 508]
	A vector measure $\mu : \A \to Y$ is called regular, if
	$$\forall A \in \A, \: \varepsilon > 0:     \quad       \exists K \in \mathcal{K}_X,\: U \in \Top_X:      \quad        K\subseteq A \subseteq U \; , \; \overset{\smile}{\mu}(U\setminus K) < \varepsilon$$
\end{definition}

To conclude this subsection we fix notation for several natural topologies on the space $\CA(\A,Y)$.
The topology of setwise convergence (i.e. on each measurable subset) is denoted by $\Top_{set}$.
The topology of convergence in variation $\Top_{v}$ is induced by the norm $\|\mu\|_{v} := \overline{\mu}(X)$.
Given an isometry $Y \hookrightarrow L(\V, \W)$, the topology of convergence in semivariation $\Top_{sv}$ is induced by the norm $\|\mu\|_{sv} := \widehat{\mu}(X)$.
Separate notation $\Top_{ssv}$ is reserved for the topology induced by the scalar semivariation.

\subsection{Measurable functions}\label{8}

In this subsection $Y \in \{\V, \overline{\mathbb{R}}\}$.

Let $\mathcal{M}(\A, \mathcal{B}_Y) := \{f:X \to Y \:|\: \forall B\in \mathcal{B}_Y: f^{-1}(B)\in \A\}$ be the set of all measurable (in a basic sense) functions from $X$ to $Y$. 
The characteristic function of a measurable set $A \in \A$ is denoted by $\chi_{A}$.
In a usual way we define the set of all $\A$-simple and $\A$-elementary $Y$-valued functions:
$$\Sim(\A, Y) := \big\{ f \in \mathcal{M}(\A, \mathcal{B}_Y) \::\: |\im f| < \aleph_0 \big\}$$
$$\mathcal{E}(\A, Y) := \big\{ f \in \mathcal{M}(\A, \mathcal{B}_Y) \::\: |\im f| \leq \aleph_0 \big\}$$

\begin{theorem}
	$\mathcal{M}(\A, \mathcal{B}_Y)$ is closed under sequential pointwise limits.
\end{theorem}
\begin{proof}
	Suppose $\mathcal{M}(\A, \mathcal{B}_Y) \ni f_n \xrightarrow{n\to\infty} f \in Y^X$, and $F\subseteq Y$ is closed.
	Define for $n\in\mathbb{N}$ the subsets \\ $U_n := \big\{ y\in Y \::\: \mathrm{dist}(y,F) < 1/n \big\} \in \mathcal{B}_Y$.
	Then $$\displaystyle f^{-1}(F) = \bigcap_{n\in\mathbb{N}} \bigcup_{N\in\mathbb{N}} \bigcap_{k \geq N} f_k^{-1}(U_n) \:\in\: \A$$
\end{proof}

\begin{definition}
	A function $f:X\to Y$ is called strongly $\A$-measurable, if it is a pointwise limit of a sequence of $\A$-simple functions.
	The corresponding space will be denoted by $\SM(\A, Y)$.
	Note that for $f \in \SM(\A, \overline{\mathbb{R}})$ an approximation by $\Sim(\A,\mathbb{R})$ is also possible.
\end{definition}

\begin{theorem}[\autocite{Dobrakov_Panchapagesan}, Theorem 1] \label{9}
	For a function $f : X \to \V$ the following conditions are equivalent:
	\begin{enumerate}[ label=\alph*)]		
		\item
		$f \in \SM(\A, \V)$
		
		\item
		$f \in \mathcal{M}(\A, \mathcal{B}_\V)$ and $\im{f}$ is separable
		
		\item
		$\displaystyle \exists\: \{f_n\}_1^\infty \in \mathcal{E}(\A, \V)^{\mathbb{N}} \::\: f_n \overset{n \to \infty}{\rightrightarrows} f$
	\end{enumerate}
	Consequently, $\SM(\A, \V)$ is closed under sequential pointwise limits.
\end{theorem}

We will use the following notation for the modulus of a function $|f|: x \mapsto \big\|f(x)\big\|_\V$.

\begin{theorem}[\autocite{Dobrakov_Panchapagesan}, Proposition 7] \label{10}
	If $f \in \SM(\A, \V)$, then $|f| \in \SM(\A, \mathbb{R})$.
	Moreover there exists a sequence $\{f_n\}_1^\infty$ in $\Sim(\A, \V)$ such that $f_n \to f$ \:and\: $|f_n| \nearrow |f|$ pointwise.
\end{theorem}

\begin{theorem} \label{11}
	Let $F \in \SM(\A, L(\V,\W))$.
	Then there exists a sequence $\{s_n\}_1^\infty \in \Sim(\A, \V)^{\mathbb{N}}$ such that
	$$\forall x \in X \::\: \left\| F(x)\big[s_n(x)\big] \right\|_{\W} \xrightarrow{n\to\infty} \|F(x)\|_{op} \qquad\text{and}\qquad \|s_n(x)\| \in \{0,1\}$$
\end{theorem}
\begin{proof}
	According to the preceding assertion there exists a sequence of elementary functions $\{F_n\}$ in $\mathcal{E}\big[\A, L(\V, \W)\big]$, converging to $F$ uniformly.
	Passing to a subsequence if necessary, we can assume $\|F_n(x) - F(x)\|_{op} \leq 1/n$.
	By definition each function has a representation of the form:
	$$F_n = \sum_{i=1}^{\infty} \chi_{A_{n,i}} \mathcal{F}_{n,i} \quad\text{where}\quad A_{n,i} \in \A \quad X = \bigsqcup_{i=1}^{\infty}A_{n,i} \quad \mathcal{F}_{n,i} \in L(\V, \W)$$
	For each $n\in\mathbb{N}$ there exists a sequence of vectors, satisfying:
	$$\{v_{n,i} \in \V\}_{i=1}^\infty \::\: \|v_{n,i}\|=1, \quad \|\mathcal{F}_{n,i}[v_{n,i}]\| \geq \|\mathcal{F}_{n,i}\|_{op} - 1/n$$
	Define the elementary sequence $\varphi_n = \sum_{i=1}^{\infty} \chi_{A_{n,i}} v_{n_i}$.
	Then we have:
	\begin{multline*}
		\big|\|F(x)[\varphi_n(x)]\|  - \|F(x)\|_{op}\big|     \leq     \big|\|F(x)[\varphi_n(x)]\| - \|F_n(x)[\varphi_n(x)]\|\big|   +  \\   \big|\|F_n(x)[\varphi_n(x)]\| - \|F_n(x)\|_{op}\big|    +      \big|  \|F_n(x)\|_{op} - \|F(x)_{op}\| \big|    \leq 3/n
	\end{multline*}
	and therefore $\|F(x)[\varphi_n(x)]\|  \overset{n \to \infty}{\rightrightarrows}  \|F(x)\|_{op}$.
	
	 Now we construct a simple sequence $\{s_n\}_{n=1}^{\infty}$ in the following way:
	 $$B_k^n := \bigcup_{i=k}^{n} \bigsqcup_{j=1}^{n} A_{i,j} \quad B_{n+1}^{n} := \varnothing \quad 1 \leq k \leq n+1$$
	 $$
	 s_n : 
		\begin{cases}
			\begin{aligned}
				& A_{k,j} \setminus B_{k+1}^{n} \cap [F \neq 0] \xrightarrow{\quad} v_{k,j} \\
				& A_{k,j} \setminus B_{k+1}^{n} \cap [F = 0] \xrightarrow{\quad} 0
			\end{aligned}
			& \quad 1\leq k,j \leq n
		\end{cases}
	$$
	It remains to verify pointwise convergence of $\{\|F(x)[s_n(x)]\|\}_{n=1}^{\infty}$.
	Fix $\epsilon>0$ and $x\in X$.
	Take $N_1 \in \mathbb{N}$ such that $3/N_1 < \epsilon$.
	Then $\exists N_2 \in \mathbb{N} \::\: x \in A_{N_1, N_2} \subseteq \bigcap_{n \geq N_1 \lor N_2} B_1^n$.
	Fix $n \geq N_1 \lor N_2$.
	If $F(x)=0$, then by construction $s_n(x)=0$ and $\|F(x)[s_n(x)]\| = 0$.
	If $F(x) \neq 0$, let $N_1 \leq k \leq n$ be the maximal number satisfying $x \in \bigsqcup_{j=1}^{n} A_{k,j}$.
	By construction then $s_n(x) = v_{k,j}$ for some $1 \leq j \leq n$.
	Consequently
	$$\big|\|F(x)[s_n(x)]\|  - \|F(x)\|_{op}\big| = \big|\|F(x)[\varphi_k(x)]\|  - \|F(x)\|_{op}\big| \leq 3/N_1 < \epsilon$$	
\end{proof}

Now suppose we have a countably subadditive submeasure $\lambda : \A \to [0,\infty]$.
Recall that the Lebesgue completion $(\overline{\A}_\lambda, \overline{\lambda})$ is defined as $\overline{\A}_\lambda := \{A \cup N \::\: A \in \A, \: N\subseteq N' \in \A, \: \lambda(N')=0\}$, with $\lambda$ extended to the countably subadditive submeasure $\overline{\lambda}$ on $\overline{\A}_\lambda$ in an obvious way.
The set of $\lambda$-measurable $Y$-valued functions is defined as $\mathcal{M}(\lambda, Y) := \mathcal{M}(\overline{\A}_\lambda, Y)$.
A function $f:X \to Y$ is said to be $\lambda$-essentially separably valued, if $\exists N \in \A \::\: \lambda(N)=0, \: f(X \setminus N)$ is separable.

\begin{definition}
	A function $f:X\to Y$ is called strongly $\lambda$-measurable, if it is a $\lambda$-a.e limit of a sequence of $\A$-simple functions.
	The corresponding space is denoted by $\SM(\lambda, Y)$.
\end{definition}

\begin{theorem}[\autocite{Dobrakov_Panchapagesan}, Theorem 2]
	For a function $f : X \to \V$ the following conditions are equivalent:
	\begin{enumerate}[label=\alph*)]		
		\item
		$f \in \SM(\lambda, \V)$
		
		\item
		$f \in \mathcal{M}(\lambda, \mathcal{B}_\V)$ and is $\lambda$-essentially separably valued
				
		\item
		$\displaystyle \exists\: \{f_n\}_1^\infty \in \mathcal{E}(\A, \V)^{\mathbb{N}}, \: N \in \A \::\: \lambda(N)=0, \: f_n \overset{X \setminus N}{\rightrightarrows} f$
	\end{enumerate}
	Consequently, $\SM(\lambda, \V)$ is closed under sequential $\lambda$-a.e. limits.
\end{theorem}

To conclude this subsection we also fix notation for the space of bounded $\A$-measurable and $\lambda$-measurable functions:
$$\BSM(\A, Y) := \SM(\A, Y) \:\cap\: B(X,Y) \qquad \BSM(\lambda, Y) := \SM(\lambda, Y) \:\cap\: B(X,Y)$$

\subsection{Definition and basic properties of the Dobrakov integral} \label{12}
For the rest of section \ref{1} we fix an operator-valued measure $\mu : \A \to L(\V, \W)$, countably additive in the norm topology with finite semivariation $\widehat{\mu} < \infty$.
Saying that a property holds $\mu$-almost everywhere refers to any of the submeasures $\overline{\mu}, \widehat{\mu}, \widetilde{\mu}$, all of which give the same null sets.

Integral of a simple function $s\in \Sim(\A, \V)$ over a set $E \in \A$ is defined in a natural way.
If $s = \sum_{i=1}^{m} \chi_{A_i} v_i$ with $\{A_i\}_{i=1}^{m} \in \mathrm{FD}(\A)$ and $v_i \in \V$, then:
$$\int_{E} s \diff{\mu} := \sum_{i=1}^{m} \mu(A_i \cap E)[v_i] \in \W$$
It is easily verified that this integral is independent of the function's representation and is therefore well-defined.

The following fundamental lemma serves as the basis for the definition of Dobrakov integrability.
\begin{lemma}[\autocite{Dobrakov_Panchapagesan}, Theorem 7] \label{13}
	Suppose \:$\Sim(\A, \V) \ni s_n \xrightarrow{\mu\text{-a.e.}} f \in \V^X$.
	Then:
	\begin{enumerate}[label=\alph*)]
		\item
		The following conditions are equivalent:
		\begin{enumerate}[label=\roman*)]
			\item
			$\displaystyle \left\{ \int_{(\cdot)}s_n \diff{\mu} \right\}_{n=1}^\infty$ converges setwise on $\A$
			
			\item
			$\displaystyle \left\{ \int_{(\cdot)}s_n \diff{\mu} \right\}_{n=1}^\infty$ converges uniformly on $\A$
			
			\item
			The family $\displaystyle \left\{ \int_{(\cdot)}s_n \diff{\mu} \right\}_{n=1}^\infty$ is uniformly countably additive
		\end{enumerate}
		
		\item \label{14}
		If $\{s'_n\}_1^\infty$ is another sequence satisfying the above conditions, then
		$$\forall A\in\A: \: \lim_{n\to\infty} \int_{A}s_n \diff{\mu} =  \lim_{n\to\infty} \int_{A}s'_n \diff{\mu}$$
	\end{enumerate}
\end{lemma}

\begin{definition}[Dobrakov integrability, \autocite{Dobrakov_Panchapagesan}, Definition 15] \label{15}
	A function $f \in \V^X$ is Dobrakov integrable, if there exists a sequence $\{s_n\}_1^\infty \in \Sim(\A, \V)^{\mathbb{N}}$, converging $\lambda$-a.e. to f and satisfying any of the conditions (i-iii) of Lemma \ref{13}.
	Such sequences will be called approximating and the corresponding set labeled as $\mathrm{AS}(f,\mu)$.\\
	By assertion \ref{13}-\ref{14} for each $A \in \A$ we have a well-defined integral $\int_{A} f \diff{\mu} := \lim_{n\to\infty} \int_{A}s_n \diff{\mu}$.\\
	The space of integrable functions is denoted by $\mathcal{I}_D(\mu,\V) := \left\{f \in \V^X : \mathrm{AS}(f,\mu) \neq \varnothing \right\}$.\\
	Obviously, $\mathcal{I}_D(\mu,\V)$ forms a linear subspace of $\SM(\mu, \V)$ over $\F$ .
\end{definition}

Of high importance is the next result, asserting for each integrable function the existence of an approximating sequence of a certain kind.
In turn it also yields a useful representation of the semivariation in terms of the integral.

\begin{theorem}[\autocite{Dobrakov_Panchapagesan}, Theorem 8] \label{16}
Suppose $f \in \mathcal{I}_D(\mu,\V)$ and $A \in \A$. Then:
	\begin{enumerate}[label=\alph*)]
		\item
		$\displaystyle \exists \{s_n\}_1^\infty \in \mathrm{AS}(f, \mu) \::\:  |s_n| \nearrow |f| \quad\mu\text{-a.e.}$
		
		\item
		$\displaystyle \widehat{\mu}(A)	   = 	\sup \left\{ \left\| \int_{A} s \diff{\mu} \right\|_{\W} :\: s \in \Sim(\A, \V), \: \|s\|_A \leq 1 \right\}    =	  \sup \left\{ \left\| \int_{A} g \diff{\mu} \right\|_{\W} :\: g \in \mathcal{I}_D(\mu,\V), \: \|g\|_A \leq 1 \right\}$
	\end{enumerate}
\end{theorem}

Some basic properties of the integral are summarized in the following theorem.

\begin{theorem}[\autocite{Dobrakov_1}, Theorem 3, Theorem 14] \label{17}
	The operation $\displaystyle \int_{(\cdot)} (\cdot) \diff{\mu} : \A \times \mathcal{I}_D(\mu,\V) \to \W$ \: has the following properties:
	\begin{enumerate}[label=\alph*)]
		\item
		$\displaystyle{  \int_{A} (\cdot) \diff{\mu} \:\in\: \mathrm{Lin}_{\F}[\mathcal{I}_D(\mu,\V), \W]   }$
		
		\item
		$\displaystyle{   \int_{(\cdot)} f \diff{\mu} \:\in\: \CA[ \A, \W ] , \quad  \int_{(\cdot)} f \diff{\mu} \:<<\:  \tilde{\mu} \leq \widehat{\mu} }$ 
		
		\item
		$\displaystyle{  \forall f \in \mathcal{I}_D(\mu,\V), \: A \in \A \::\: \left\| \int_{A} f \diff{\mu} \right\|_{\W} \leq \|f\|_A \cdot \widehat{\mu}(A) }$
	\end{enumerate}
\end{theorem}

We will also need the following result describing how the integral interacts with bounded linear operators.

\begin{theorem}[\autocite{Dobrakov_1}, p. 526] \label{18}
	Let $\W'$ be another Banach space and $T \in L(\W, \W')$. Then:
	\begin{enumerate}[label=\alph*)]
		\item
		$\mathcal{I}_D(\mu,\V) \subseteq \mathcal{I}_D(T\mu, \V)$		
		
		\item
		$\displaystyle{  \forall f \in \mathcal{I}_D(\mu,\V), \: A \in \A \::\:  \int_{A} f \diff{[T\mu]} = T\left[\int_{A} f \diff{\mu}\right] }$
	\end{enumerate}
\end{theorem}

%
%
%

\subsection{Dobrakov semivariation and $L_p$ spaces}

\begin{definition}[\autocite{Dobrakov_2}, Definition 1', p. 694]
	For $1 \leq p < \infty$ the Dobrakov $p$-semivariation of a function $f \in \SM(\mu, [0,\infty])$ on a set $A \in \A$ is defined as:
	$$\widehat{\mu}_p(f, A) := \sup \left\{ \left\| \int_{A} s \diff{\mu} \right\|_{\W}^{1/p} \::\: s \in \Sim(\A, \V), \: \|s(\cdot)\| \leq f^p \right\}$$
	For notational consistency we also define $\widehat{\mu}_\infty$ as essential supremum:
	$$\widehat{\mu}_\infty(f, A) = \esssup\big[f, A, \widehat{\mu}\big] := \inf \big\{ C \geq 0 \::\: f \leq C \;\widehat{\mu}\text{-a.e. on } A \big\}$$
	For $f \in \SM(\mu, \V), \: 1\leq p \leq \infty$ we let $\widehat{\mu}_p(f, A) := \widehat{\mu}_p(|f|, A)$.\\
	The Dobrakov $L_p$-seminorm is defined as $\|f\|_p := \widehat{\mu}_p(f, X)$
\end{definition}

Noticeably, for $1 \leq p < \infty$ the Dobrakov semivariation extends the classical semivariation, since clearly $\widehat{\mu}(A) = \widehat{\mu}_1(1,A)$.
The following equivalent characterizations of $p$-semivariation will be used later on.

\begin{theorem}[\autocite{Dobrakov_2}, Theorems 2, 4] \label{21}
	Suppose $f \in \SM(\mu, [0,\infty])$, \: $A \in \A$ \:and\: $1 \leq p < \infty$. Then:
	\begin{enumerate}[label=\alph*)]
		\item
		$\displaystyle  \widehat{\mu}_p(f, A) = \sup \left\{ \left\| \int_{A} g \diff{\mu} \right\|_{\W}^{1/p} \::\: g \in \mathcal{I}_D(\mu,\V), \: \|g(\cdot)\| \leq f^p  \;\mu\text{-a.e.}\right\}$		
		
		\item
		$\displaystyle  \widehat{\mu}_p(f, A) = \sup \left\{ \svar_{\F,\W}\left( \int_{(\cdot)} g \diff{\mu} \:,\: A \right) \::\: g \in \mathcal{I}_D(\mu,\V), \: \|g(\cdot)\| \leq f^p  \;\mu\text{-a.e.}\right\}^{1/p}$		
		
		\item
		$\displaystyle  \widehat{\mu}_p(f, A) = \sup \left\{   \left(\int_{A} f^p \diff\big[\overline{w^*\mu}\big]\right)^{1/p} \::\: w^* \in \W^*, \: \|w^*\| \leq 1 \right\}$ \qquad (ordinary Lebesgue integral)
	\end{enumerate}
\end{theorem}

Basic properties of Dobrakov semivariation are summarized below.

\begin{theorem}[\autocite{Dobrakov_2}, Theorems 1, 3] \label{22}
	For $1 \leq p \leq \infty$ the Dobrakov semivariation has the following properties:
	\begin{enumerate}[label=\alph*)]
		\item
		\begin{tabular}{p{10cm} p{7cm}}
		$\widehat{\mu}_p(f, \cdot) : \A \to [0,\infty]$  \; is a $\sigma$-subadditive submeasure 					& $f \in \SM(\mu, [0,\infty])$
		\end{tabular}
				
		\item
		\begin{tabular}{p{10cm} p{7cm}}
		$f \leq g \;\mu\text{-a.e.} \;\implies\; \widehat{\mu}_p(f, \cdot) \leq \widehat{\mu}_p(g, \cdot)$ 			& $ f,g \in \SM(\mu, [0,\infty])$
		\end{tabular}
		
		\item
		\begin{tabular}{p{10cm} p{7cm}}
		$\widehat{\mu}_p(f, A) = 0 \;\iff\; f=0 \;\mu\text{-a.e}\text{ on } A$ 							& $f \in \SM(\mu, [0,\infty]), \: A \in \A$
		\end{tabular}
		
		\item
		\begin{tabular}{p{10cm} p{7cm}}
		$\widehat{\mu}_p(\alpha \cdot f, A) = |\alpha| \cdot \widehat{\mu}_p(f, A)$ 						& $\alpha \in \F, \: f \in \SM(\mu, [0,\infty]), \: A \in \A$
		\end{tabular}
		
		\item
		\begin{tabular}{p{10cm} p{7cm}}
		$\widehat{\mu}_p(f+g, A) \:\leq\:  \widehat{\mu}_p(f, A) \:+\: \widehat{\mu}_p(g, A)$						& $f,g \in \SM(\mu, [0,\infty]), \: A \in \A$
		\end{tabular}
		
		\item
		\begin{tabular}{p{10cm} p{7cm}}
		$\displaystyle \inf_{x\in A} f(x) \: \widehat{\mu}(A)^{1/p} \:\leq\: \widehat{\mu}_p(f, A) \:\leq\: \sup_{x\in A} f(x) \: \widehat{\mu}(A)^{1/p}$			& $f \in \SM(\mu, [0,\infty]), \: A \in \A$
		\end{tabular}
	\end{enumerate}
\end{theorem}

\begin{theorem}[H\"older inequality] \label{23}
	Suppose $ f,g \in \SM(\mu, [0,\infty])$, \: $1 \leq p < \infty$, \: $p^{-1} + q^{-1} = 1$. \\
	Then: \; $\widehat{\mu}_1(fg, \cdot)    \leq      \widehat{\mu}_p(f, \cdot) \: \widehat{\mu}_q(g, \cdot)$
\end{theorem}
\begin{proof}
	If $1<p<\infty$, using Theorem \ref{21} and classical H\"older inequality, we obtain:
	\begin{align*}
	\widehat{\mu}_1(fg, A) & = \sup \left\{   \int_{A} fg \diff\big[\overline{w^*\mu}\big] \::\: w^* \in \W^*, \: \|w^*\| \leq 1 \right\} \\
	& \leq \sup \left\{   \left(\int_{A} f^p \diff\big[\overline{w^*\mu}\big]\right)^{\frac{1}{p}}    \left(\int_{A} g^q \diff\big[\overline{w^*\mu}\big]\right)^{\frac{1}{q}}       \::\:      w^* \in \W^*, \: \|w^*\| \leq 1 \right\} \\
	& \leq \widehat{\mu}_p(f, A) \cdot \widehat{\mu}_q(g, A)
	\end{align*}
	Bearing in mind that $\esssup\big[g, A, \overline{w^*\mu}\big] \leq \esssup\big[g, A, \widehat{\mu}\big]$, the case $p=1, q=\infty$ is proved similarly:
	\begin{align*}
	\widehat{\mu}_1(fg, A) & = \sup \left\{   \int_{A} fg \diff\big[\overline{w^*\mu}\big] \::\: w^* \in \W^*, \: \|w^*\| \leq 1 \right\} \\
	& \leq \sup \left\{  \esssup\big[g, A, \overline{w^*\mu}\big]   \cdot    \int_{A} f \diff\big[\overline{w^*\mu}\big]        \::\:      w^* \in \W^*, \: \|w^*\| \leq 1 \right\} \\
	& \leq \widehat{\mu}_1(f, A) \cdot \widehat{\mu}_\infty(g, A)
	\end{align*}

\end{proof}

Immediately from Theorem \ref{21} and H\"older inequaility \ref{23} we get the following upper bound:
$$\left\| \int_A f \diff{\mu} \right\|_{\W} \:\leq\: \widehat{\mu}_p(f,A) \cdot \widehat{\mu}(A)^{1/q}    \qquad     f \in \mathcal{I}_D(\mu, \V), \; A \in \A, \; 1 \leq p \leq \infty$$

\begin{lemma}[Fatou's Lemma, \autocite{Dobrakov_5}, Lemma 1] \label{24}
	Suppose $\{f_n\}_1^\infty$ is a sequence in $\SM(\mu, [0,\infty])$. Then:
	$$ \forall A \in \A: \quad \widehat{\mu}_1\big(\liminf_{n\to\infty} f_n , A\big) \:\leq\: \liminf_{n\to\infty} \widehat{\mu}_1\big(f_n,A\big)$$
\end{lemma}

\begin{theorem}[\autocite{Dobrakov_2}, Theorem 5] \label{25}
	For $f \in \mathcal{I}_D(\mu, \V)$ continuity of $\widehat{\mu}_1(f, \cdot)$ is equivalent to the existence of the following control measure:
	$$\lambda \in \CA\big[ \A, [0,\infty) \big] :    \quad         \lambda   \;\leq\;   \svar_{\F, \W}\left( \int_{(\cdot)} f \diff{\mu}  \:,\: \cdot \right)   \;\leq\;   \widehat{\mu}_1(f, \cdot)     \;<<\;    \lambda$$
	Moreover, continuity of $\widehat{\mu}_1(f, \cdot)$ implies $\widehat{\mu}_1(f, X) < \infty$.
\end{theorem}

According to \autocite[Theorem 7]{Dobrakov_2}, the space $\SM(\mu, \V)$ with the $L_p$-seminorm forms a complete extended pseudometric space.
Identifying functions equal $\mu$-a.e. we obtain a complete extended metric space.
Naturally, the next step is to restrict one's attention to those strongly measurable maps that have a finite seminorm.
In integration schemes with absolute integrability (such as Bochner and Dinculeanu integrals) this condition implies integrability and determines the $L_p$ spaces.
In Dobrakov's theory the situation is more complex.
In general there may be integrable functions with infinite $L_1$-seminorm, as well as strongly measurable functions with finite $L_1$-seminorm, that are not integrable \autocite[Example 1]{Panchapagesan_ODF}.
Consequently, there arise naturally not one but four spaces of functions with finite seminorm, each of which in a certain way corresponds to the classical Lebesgue space.
The reader is referred to \autocite{Dobrakov_2} and \autocite{Panchapagesan_ODF} for details.
In this paper we will be concerned only with the smallest and most well-behaved of these spaces, the definition of which is formulated below.

\begin{definition}[Dobrakov $L_p$ space, \autocite{Dobrakov_2}, Definition 4]
	Let $Y \in \big\{\V, \: [0,\infty]\big\}$.
	The Dobrakov $L_p$ space for $1\leq p< \infty$ is defined as follows:
	$$\mathcal{L}_{p}(\mu, Y) \::=\: \big\{f \in \SM(\mu, Y) \::\: \widehat{\mu}_p(f, \cdot) \text{ is continuous on  } \A \big\}$$
	The corresponding quotient space after identifying functions equal $\mu$-a.e. is denoted by $L_{p}(\mu, Y)$.
\end{definition}

According to \autocite[Theorem 7]{Dobrakov_2}, $\mathcal{L}_{p}(\mu, \V)$ is a closed subspace of $\SM(\mu, \V)$ in the $L_p$-seminrom.
Therefore the quotient $\big( L_{p}(\mu, \V), \|\cdot\|_p \big)$ is a Banach space.
Moreover, for these spaces there are complete analogs of Vitali and Lebesgue convergence theorems \autocite{Dobrakov_2} and the Monotone Convergence Theorem \autocite[Theorem 2]{Dobrakov_5}, and only these spaces can be separable \autocite[Theorem 19]{Dobrakov_2}.
It should be noted that even when the measure has finite variation the $\mathcal{L}_1$ class of Dobrakov already contains the space of Dinculeanu integrable functions.
In some cases the inclusion may be strict \autocite{Panchapagesan_ODF}.

\begin{proposition} \label{26}
	$1 \leq p \leq r < \infty \quad\implies\quad \mathcal{L}_{r}(\mu, \V) \subseteq \mathcal{L}_{p}(\mu, \V)$.
\end{proposition}
\begin{proof}
	Fix $f \in \SM(\mu,\V)$.
	If $p<r$, take the conjugate exponents $r/p$ and $r/(r-p)$ and apply \ref{23}:
	$$\widehat{\mu}_p(f,\cdot)    \;=\;    \widehat{\mu}_1(|f|^p, \cdot)^{1/p}       \;\leq\;        \Big(  \widehat{\mu}_{r/p}(|f|^p, \cdot) \; \widehat{\mu}_{r/(r-p)}(1, \cdot) \Big)^{1/p}       \;\leq\;     \widehat{\mu}_{r}(f, \cdot) \; \widehat{\mu}(X)^{(r-p)/r} $$
	Now continuity of $\widehat{\mu}_r(f,\cdot)$ implies continuity of $\widehat{\mu}_p(f,\cdot)$.
\end{proof}

The next proposition asserts that the integral measure defined by a nonzero $\mathcal{L}_1$-density cannot vanish on every set.
If the density is merely from $\mathcal{I}_D(\mu,\V)$, this may not be the case in general.

\begin{proposition} \label{27}
	For $f \in \mathcal{L}_1(\mu, \V)$ we have \:$\mu f = 0    \iff    f=0 \;\:\widehat{\mu}\text{ - a.e.}$
\end{proposition}
\begin{proof}
	The statement follows from the existence of control measure described in Theorem \ref{25}.
	If the integral measure $\mu f$ vanishes on each set, the same must hold for the control measure, and hence for $\widehat{\mu}_1(f,\cdot)$ by absolute continuity.
	By Theorem \ref{22} this means $f=0 \;\:\widehat{\mu}\text{ - a.e.}$
\end{proof}

Finally we state here the aforementioned analog of the Lebesgue Dominated Convergence Theorem, that will be used numerous times throughout the paper.
\begin{theorem}[Lebesgue Dominated Convergence, \autocite{Dobrakov_2}, Theorem 17] \label{28}
	Suppose \: $\displaystyle \SM(\mu, \V) \:\ni\: f_n  \:\xrightarrow{\mu\text{-a.e.}}\: f \:\in\: \SM(\mu, \V)$ \;and\;  $|f_n| \leq g \in \mathcal{L}_{p}(\mu, [0,\infty])$. \\
	Then: \: $\displaystyle \mathcal{L}_{p}(\mu, \V) \:\ni\: f_n  \:\xrightarrow{L_p}\: f \:\in\: \mathcal{L}_{p}(\mu, \V)$
\end{theorem}

\subsection{Operator-valued densities and change of measure formula} \label{29}
Let $\V'$ be another Banach space.
Consider the following linear isometry:
$$L(\V, \W) \hookrightarrow L\big[ L(\V', \V) , L(\V', \W)\big]$$
$$L(\V, \W) \:\ni\: A \:\mapsto\:  \big[ L(\V', \V) \:\ni\: B \:\mapsto\:  A\circ B    \big] \:\in\: L\big[ L(\V', \V) , L(\V', \W) \big]$$
This mapping allows to interpret our measure $\mu$ as taking values in the space $L\big[ L(\V', \V) , L(\V', \W) \big]$.
Denote this measure by $M$. Its semivariation $\widehat{M}$ is by definition the semivariation of $\mu$ relatively to the introduced isometry.
It is easily verified that $M$ is also countably additive in the uniform operator topology.
Therefore we can speak of integrating $L(\V', \V)$-valued functions with respect to $M$.
By Theorem \ref{17} such integration results in $L(\V', \W)$-valued measures, countably additive in the uniform operator topology and absolutely continuous with respect to $\widehat{M}$.
But beforehand it is necessary to make sure that $M$ has finite semivariation, which follows from a more general statement below.

\begin{theorem}[Coincidence of Dobrakov semivariations] \label{30}
	$\widehat{M}_p = \widehat{\mu}_p$ for all $1 \leq p \leq \infty$
\end{theorem} 
\begin{proof}
	First suppose $1 \leq p < \infty$.
	Fix $f \in \SM(\A, [0,\infty])$ and $A \in \A$.\\
	Let $s = \sum_{i=1}^{n} \chi_{A_i} H_i \:\in\: \Sim\big[ \A, L(\V', \V) \big], \; \|s(\cdot)\| \leq f^p$. Then:
	\begin{align*}
		\left\| \int_{A} s \diff{M} \right\|_{op}      &=       \sup_{\|v\| \leq 1} \left\| \sum_{i=1}^n M(A \cap A_i)[H_i]v \right\|_{\W}      \:=\:       \sup_{\|v\| \leq 1} \left\| \sum_{i=1}^n \mu(A \cap A_i)[H_i v] \right\|_{\W} \\
		& \leq \sup \left\{ \left\| \int_{A} s' \diff{\mu} \right\|_{\W} \::\:  s' \in \Sim(\A, V), \::\: \|s'(\cdot)\| \leq f^p \right\}     \:\leq\:   \widehat{\mu}_p(f, A)^p
	\end{align*}
	Consequently we obtain $\widehat{M}_p(f, A) \leq \widehat{\mu}_p(f,A)$.\\
	For the reverse inequality, let $s = \sum_{i=1}^{n} \chi_{A_i} v_i \:\in\: \Sim\big[ \A, \V \big], \; \|s(\cdot)\| \leq f^p$.
	Choose a unit vector $v' \in \V'$ and a functional $\varphi \in \mathbb{V'}^*$ such that $\|\varphi\| = \varphi(v') = 1$.
	Define the operators $H_i := \varphi \cdot v_i $ and $s' = \sum_{i=1}^{n} \chi_{A_i} H_i$.
	Then $|s'| \leq f^p$ and we have:
	$$ \left\| \int_{A} s \diff{\mu} \right\|_{\W}   =  \left\| \sum_{i=1}^n \mu(A \cap A_i) v_i \right\|_{\W} = \left\| \left(\sum_{i=1}^n M(A \cap A_i) H_i\right) v' \right\|_{\W}  \leq \left\| \int_{A} s' \diff{M} \right\|_{op} \leq  \widehat{M}_p(f, A)$$
	As a special case when $p=1$ we obtain $\widehat{M} = \widehat{\mu}$, which in turn yields the desired equality for $p=\infty$.
\end{proof}

Now we fix an integrable operator-valued function $g \in \mathcal{I}_D \big[M, L(\V', \V) \big]$ with the corresponding integral measure $\gamma = Mg := \int_{(\cdot)} g \diff{M} \allowbreak \in \CA\big[ \A, L(\V', \W) \big]$.
Our aim is to show that under reasonable conditions the change of measure formula is valid for integration with respect to $\gamma$.
This technical fact will be crucial for working with operator-valued logarithmic derivatives later on.

\begin{lemma} \label{31}
	Suppose $f \:\in\: \SM\big(\mu, \: [0,\infty]\big)$ \:and\: $1\leq p <\infty$. Then:
	$$\widehat{\mu}_p \big(|g| f, A\big)    \;=\;    \sup \Big\{ \widehat{\mu}_p \big(g[s], A\big) \::\: s \:\in\: \Sim(\A,\V'), \: |s| \leq f \Big\} ,   \quad A \in \A$$
\end{lemma}
\begin{proof}
	According to \ref{10} and \ref{11} there exist sequences:
	$$\{f_n\}_1^\infty \:\in\: \Sim\big[\A, [0,\infty)\big] \::\: f_n \nearrow f  \;\:\widehat{\mu}\text{ - a.e.}          \qquad\quad          \{s_n\}_1^\infty \:\in\: \Sim\big[\A, \V' \big] \::\: \big|s_n\big|\leq 1, \; \big|g s_n\big| \xrightarrow{\widehat{\mu}\text{ - a.e.}} \big|g\big|$$
	Using Fatou's Lemma \ref{24}, we get:
	\begin{multline*}
		\widehat{\mu}_p \big(|g| f, A\big)     \;=\;     \widehat{\mu}_1 \big(|g|^p f^p, A\big)^{1/p}     \;\leq\;     \liminf_{n\to\infty} \: \widehat{\mu}_1 \big(|g[f_ns_n]|^p, A\big)^{1/p}     \;\leq\;    \\
		 \leq\;    \sup \Big\{ \widehat{\mu}_p \big(g[s], A\big) \::\: s \:\in\: \Sim(\A,\V'), \: |s| \leq f \Big\}     \;\leq\;   \widehat{\mu}_p \big(|g| f, A\big)
	\end{multline*}
\end{proof}

\begin{lemma} \label{32}
	\hfill
	\vspace{-25pt}	
	\begin{enumerate}[label=\alph*)]
	\item
	$\displaystyle \forall s \in \Sim(\A, \V'): \quad gs \in \mathcal{I}_D(\mu, \V),      \quad      \int_{(\cdot)} gs \diff\mu = \int_{(\cdot)} s \diff\gamma$	

	\item
	$\displaystyle \forall (v', w^*) \in \V' \times \W^*: \quad       gv' \in \mathcal{I}_D(w^*\mu, \V),     \quad      \int_{(\cdot)} gv' \diff[w^*\mu] \:=\: w^*\left[  \left( \int_A g \diff M \right)v' \right] $	
	\end{enumerate}
\end{lemma}
\begin{proof}
	Let $\{g_n\}_1^\infty \in \mathrm{AS}(g, M)$ be an approximating sequence in the sense of Definition \ref{15}.
	Take $v' \in \V'$.
	Since $\widehat{\mu} = \widehat{M}$, we have $\Sim(\A, \V) \ni g_n v' \xrightarrow{\widehat{\mu}\text{ - a.e.}} gv'$.
	We can also view $v'$ as a bounded operator $v' : L(\V', \W) \ni H \mapsto Hv' \in \W$.
	Applying Theorem \ref{18} we get:
	$$\int_A g_n v' \diff\mu     \:=\:     \left( \int_A g_n \diff M \right)v'    \quad\xrightarrow{n\to\infty}\quad     \left( \int_A g \diff M \right)v'     \:=\:      \int_A g \diff[v'M] $$
	where the equality on the left is evident for simple functions, and passing to the limit is justified by continuity of the operator.
	Thus, by Definition \ref{15} we get $gv' \in \mathcal{I}_D(\mu, \V)$. \\
	Now if $s = \sum_{i=1}^{m} \chi_{A_i} v'_i \:\in\: \Sim(\A, \V')$, then $gs = \sum_{i=1}^{m} \chi_{A_i} g v'_i \:\in\: \mathcal{I}_D(\mu, \V)$ and :
	$$\int_A s \diff\gamma     \:=\:      \sum_{i=1}^{m} \gamma(A \cap A_i) v'_i       \:=\:      \sum_{i=1}^{m} \left( \int_{A \cap A_i} g \diff M \right) v'_i    \:=\:   \sum_{i=1}^{m} \int_{A \cap A_i} g v'_i \diff\mu     \:=\:    \int_{A} g s \diff\mu$$
	This proves the first assertion of the lemma. \\
	Now let us take $(v', w^*) \in \V' \times \W^*$.
	In an obvious way the pair $(v', w^*)$ can be thought of as an element of $L(\V', \W)^*$.
	Again, applying Theorem \ref{18} and taking into account that $\widehat{w^*\mu} \leq \widehat{\mu}$, we obtain:
	$$\Sim(\A, \V)    \:\ni\:     g_n v'     \:\xrightarrow{w^*\mu\text{ - a.e.}}\:      gv'       \:\in\:          \mathcal{I}_D(w^*\mu, \V)$$
	$$\int_A g_n v' \diff[w^*\mu]    \:=\:    w^*\left[  \left( \int_A g_n \diff M \right)v' \right]     \quad\xrightarrow{n\to\infty}\quad     w^*\left[  \left( \int_A g \diff M \right)v' \right]      \:=\:      \int_A g \diff[(v',w^*)M] $$	
	This establishes the second assertion by Lemma \ref{13} and Definition \ref{15}.
\end{proof}

\begin{theorem}[Change of measure formula] \label{33}
	Suppose $f \in \SM(\widehat{\mu}, \V')$ \:and\: $1 \leq p < \infty$, then:
	\begin{enumerate}[label=\alph*)]
		\item
		$\displaystyle \widehat{\gamma}_p(f, A)     \:=\:    \sup \left\{      \left\| \int_{A} g[s] \diff{\mu} \right\|_\W^{1/p}     \::\:    s \in \Sim(\A,\V'),    \;     |s| \leq  |f|^p   \right\}$
				
		\item
		$\displaystyle \widehat{\gamma}_p(f, \cdot)     \:\leq\:    \widehat{\mu}_1\big(|g| |f|^p, \cdot \big)^{1/p}$
		
		\item
		$\displaystyle |g| |f| \in \mathcal{L}_1(\mu)   \quad\implies\quad     f \in  \mathcal{L}_1(\gamma, \V')      \;,\quad       \int_{(\cdot)}f \diff\gamma = \int_{(\cdot)}gf \diff\mu$
	\end{enumerate}
\end{theorem}
\begin{proof}
	The first two statements follow from \ref{31} and \ref{32} :
	\begin{multline*}
		\widehat{\gamma}_p(f, A) \;:=\;   \sup \left\{   \left\| \int_{A} s \diff{\gamma} \right\|_\W^{1/p}   \::\:s \in \Sim(\A,\V'),    \;     |s| \leq  |f|^p    \right\}    \\
		 =\;        \sup \left\{      \left\| \int_{A} g[s] \diff{\mu} \right\|_\W^{1/p}     \::\:    s \in \Sim(\A,\V'),    \;     |s| \leq  |f|^p   \right\}   \\
		 \leq\;    \sup \left\{      \widehat{\mu}_1\big( g[s], A\big)^{1/p}     \::\:    s \in \Sim(\A,\V'),    \;     |s| \leq  |f|^p   \right\}       \;=\;     \widehat{\mu}_1\big(|g| |f|^p, \cdot \big)^{1/p}
	\end{multline*}
	
	If now $|g| |f| \in \mathcal{L}_1(\mu)$, then the second statement implies that $\widehat{\gamma}_1(f, \cdot)$ is also continuous, hence $f \in  \mathcal{L}_1(\gamma, \V')$.
	By Theorem \ref{10} there exists a sequence $\{s_n\}_1^\infty \in \Sim(\A,\V')^\mathbb{N}$ such that $s_n \to f \;\mu\text{-a.e.}$ and $|s_n| \nearrow |f| \;\mu\text{-a.e.}$.
	Then $g s_n \to gf \;\mu\text{-a.e.}$ and at the same time $|g s_n| \leq |g||f| \;\mu\text{-a.e.}$.
	By the Dominated Convergence Theorem \ref{28} and Lemma \ref{32} we get:
	\begin{align*}
		\mathcal{L}_1(\gamma, \V') \ni s_n \xrightarrow{L_1} f \in \mathcal{L}_1(\gamma, \V')     \quad&\text{and}\quad      \mathcal{L}_1(\mu, \V) \ni gs_n \xrightarrow{L_1} gf \in \mathcal{L}_1(\mu, \V) \\
		\int_{A} f \diff\gamma  \;=\;  \lim_{n\to\infty} \int_{A} s_n \diff\gamma  \;&=\;   \lim_{n\to\infty} \int_{A} g s_n \diff\mu \;=\; \int_{A} g f \diff\mu
	\end{align*}
\end{proof}

Thus it has been verified, that under reasonable assumptions on the functions the Dobrakov integral with respect to $\gamma$ behaves as expected, and the change of measure formula is valid.
In the sequel when talking about integration of operator-valued functions we will allow ourselves not to mention explicitly the transition to measure $M$ and write $\mu g$ to mean $M g$, implying the construction described in this section.

It should be noted that the question of existence of a Dobrakov density of one measure relatively to another is rather complex and lies beyond the scope of this paper. 
A Radon-Nikodym theorem for the Dobrakov integral in the form of necessary and sufficient conditions for a pair of measures is obtained in \autocite{Maynard_AGRNT}.
However these conditions seem challenging to verify in practice.

\newpage
\section{Continuity and differentiablity of vector measures} \label{34}
The theory of differentiable measures has its origins in the works of Fomin and was suggested by him as a basis for infinite-dimensional generalization of Sobolev-Schwartz distribution theory \autocite[p. ix]{Bogachev_DM}.
An introduction to the main concepts of the theory at a high level of generality can be found in \autocite{Daletsky_Fomin},
while the monograph \autocite{Bogachev_DM} contains an in depth treatment of real-valued differentiable measures on locally convex spaces.
At the same time publications dedicated to continuous or differentiable vector measures are quite rare.
For that reason the main purpose of this section is to transfer a certain set of results to the setting of vector and operator-valued measures and Dobrakov integration. 

As desribed in \autocite{Daletsky_Fomin, Bogachev_DM} the notions of continuity and differentiability for real-valued measures may be introduced relatively to various topologies, such as the topologies of setwise convergence and convergence in variation.
In this particular case these two topologies turn out to be equivalent for both concepts.
In the vector-valued context there also arises naturally the topology of convergence in semivariation, and as demonstrated in \autocite{Romanov_cont, Romanov_diff}, these three topologies in general specify pairwise non equivalent conditions of continuity and differentiability.
We will be mainly concerned with topologies of setwise and semivariational convergence.

\subsection{General notions}

Throughout section \ref{34} we set $(X,\A) = (\mathbb{R}^d, \mathcal{B}_{\mathbb{R}^d})$.
The Lebesgue measure on $X$ is denoted by $Leb_X$ or $\diff{x}$.
For manipulating mixed partial derivatives it will be convenient to introduce the following sets of multivectors and finite sequences on $X$:

\begin{tabular}{l l l}
	$M^kX := \big\{ h: X \to \omega \:\big|\: |h| = k \big\}$ 				&\qquad\qquad&		$S^kX := X^{\{1 \cdots k\}}$									\\
	$M_0^kX  := \big\{ h: X \to \omega \:\big|\: 0 \leq |h|\leq k \big\}$		&\qquad\qquad&		$S_0^kX := \big\{ s: \{1 \cdots n\} \to X \:\big|\: 0 \leq n \leq k \big\}$		\\
	$MX := \big\{ h: X \to \omega \:\big|\: |h|<\infty \big\}$				&\qquad\qquad&		$SX := \big\{ s: \{1 \cdots n\} \to X \:\big|\: n \in \omega \big\}$
\end{tabular}

For two multivectors $\alpha, \beta \in MX$ define the combinatorial coefficient:
$$C(\alpha,\beta) := \prod_{x \in X} \frac{(\alpha(x)+\beta(x))!}{\alpha(x)! \: \beta(x)!}$$
Its combinatorial meaning is the following: if $(\alpha$, $\beta$) is a pair of multisets and every element of $\alpha+\beta$ is assigned a unique number, there are $C(\alpha,\beta)$ ways to split these numbers into two groups, so that the resulting pair of multisets is still $(\alpha,\beta)$.

As before $Y, \V, \W$ are Banach spaces over $\F \in \{\mathbb{R}, \mathbb{C}\}$.
For notational brevity in this section we let  $\mathbb{M} := \CA(\A,Y)$.
We also assume a fixed linear isometry $Y \hookrightarrow L(\V,\W)$ with the corresponding semivariation denoted by $\widehat{\mu}, \: \mu \in \mathbb{M}$.

Due to the translational invariance of $\A$ the following operations are correctly defined:
$$T \::\: \M \times X    \:\ni\:  (\mu, x)  \:\mapsto\:    \mu_x := \big[ A \mapsto \mu(A+x) \big]  \:\in\:  \M$$
$$T_\mu := T(\mu, \cdot) \::\: X \ni x \mapsto \mu_x \in \M \qquad T_{\mu, A} \::\: X \ni x \mapsto \mu(A+x) \in Y $$

For a fixed measure $\mu \in \M$ it is possible to examine continuity, differentiability and smoothness of $T_\mu : X \to \M$ in some topology $\Top$ on $\M$.
Thus we are led to the following defintions.

\begin{definition}[Differentiability of a measure, \autocite{Daletsky_Fomin}, ch. IV.2]
	Let $\Top \in \big\{\Top_{set}, \Top_{sv}, \Top_{v} \big\}$.
	\begin{enumerate}[label=\alph*)]
	\item
	Measure $\mu \in \M$ is called Fomin differentiable in topology $\Top$, if the mapping $T_\mu : X \to (\M, \Top)$ is everywhere Frechet differentiable \autocite[p. 331]{Bogachev_Smolyanov_Sobolev}.
	The space of such measures is denoted by $\mathcal{D}(Y, \Top)$.
	
	\item
	Analogously, measure $\mu \in \M$ is called Fomin differentiable along vector $h \in X$ in topology $\Top$, if the mapping $T_\mu : X \to (\M, \Top)$ is everywhere differentiable along $h$.
	The space of such measures is denoted by $\mathcal{D}(Y, \Top, h)$.
	\end{enumerate}
\end{definition}

It is easy to see that in any of the above topologies differentiablity of $\mu$ is equivalent to differentiability of $T_\mu$ at zero:
	$$ \mu \in \mathcal{D}(Y, \Top) \iff T_{\mu} \in \mathcal{D}(0; \Top)   \qquad      \mu \in \mathcal{D}(Y, \Top, h) \iff \exists \: \partial_hT_{\mu}(0) \text{in } \Top$$ 

\begin{proposition}[Differentiability along a vector in $\Top_{set}$]
	For $\mu \in \M$ and $h \in X$ the following conditions are equivalent:
	\begin{enumerate}[label=\alph*)]
		\item
		$\displaystyle \forall A \in \A: \: \exists \:\lim_{t\to\infty} \big( \mu(A+th) - \mu(A) \big)/t \: =: \: \partial_h\mu(A)$
		
		\item
		$\displaystyle \mu \in \mathcal{D}(Y, \Top_{set}, h)$
	\end{enumerate}
	Measure $\partial_h\mu$ is called the Fomin derivative of $\mu$ along vector $h$.
\end{proposition}
\begin{proof}
	If the limit in (a) exists for every set, then the Vitali-Hahn-Saks-Nikodym Theorem \ref{3} assures that $\partial_h\mu \in \M$.
	By definiton $\partial_h\mu = \partial_h T_\mu(0)$.
	Next we can obtain differentiability of $T_\mu$ along $h$ at any other point:
	$$\frac{T_\mu(x+th) -T_\mu(x)}{t}[A]  \:=\:   \frac{\mu(A+x+th) - \mu(A+x)}{t}      \;\xrightarrow{t \to 0}\;     \partial_h\mu(A+x) $$
	The reverse implication is obvious.
\end{proof}

Inductively we can define (in case of existence) mixed Fomin derivatives of higher order.
The next assertion follows by analogous reasoning, so the proof is omitted.

\begin{proposition}[Mixed derivatives in $\Top_{set}$]
	Suppose $\mu \in \M$  \:and\: $s \:\in\: S^kX$.
	Define \: $\partial_{s} := \partial_{s(1)} \cdots \partial_{s(k)}$.\\
	Then the following conditions are equivalent:
		\begin{enumerate}[label=\alph*)]
			\item
			$\displaystyle \exists \: \partial_{s}T_{\mu}(0) \:\in\: \mathbb{M} \:\text{ in topology } \Top_{set}$
			
			\item
			$\displaystyle \exists \: \partial_{s}T_{\mu} : X \to \mathbb{M} \:\text{ in topology } \Top_{set}$
			
			\item
			$\displaystyle \forall A \in \A : \; \exists \: \partial_{s}T_{\mu,A}(0) \:\in\: Y$
			
			\item
			$\displaystyle \forall A \in \A : \; \exists \: \partial_{s}T_{\mu,A}: X \to Y$
			
			\item
			$\displaystyle \exists \: \partial_{s}\mu \in \mathbb{M}$
		\end{enumerate}
		In this case relatively to topology $\Top_{set}$ we have the expressions:
		$$T\big( \partial_{s}\mu , x \big) \:=\: \partial_{s}T_{\mu}(x)        \qquad\text{and}\qquad        \partial_{s}T_{\mu}(x)[A]   \:=\:   \partial_{s}T_{\mu,A}(x)    \:=\:     \partial_{s}\mu(A+x)$$
\end{proposition}

Now, bearing in mind the construction in subsection \ref{29}, let us again consider the idea of reinterpreting an operator-valued measure as acting on operators via left composition.
More precisely we have the following linear operator, that is a $\|\cdot\|_{sv}$-isometry by \ref{30} and can be seen as a form of lifting: 
$$\Lift \::\: \CA\big[\A,  L(\V,\W) \big]      \:\to\:       \CA\left[\A,  L\big[ L(\V), L(\V,\W) \big] \right]$$
$$\Lift(\mu)    \;:\;     \A \ni A      \;\mapsto\;     \big[  L(\V) \ni F \mapsto \mu(A)\circ F \in L(\V,\W) \big]      \;\in\;      L\big[ L(\V), \: L(\V,\W) \big]$$

The next proposition verifies that, as expected, $\Lift$ commutes with translations and Fomin differentiation.
\begin{proposition} \label{35}
	Let $\mu \in \CA\big[\A,  L(\V,\W) \big]$.
	Then:
	\begin{enumerate}[label=\alph*)]
	\item
	$\displaystyle \forall x \in X : \quad   T\big( \Lift\mu, x \big) = \Lift \big[T( \mu, x )\big]$
	
	\item
	$\displaystyle \forall s \in SX : \quad  \exists\: \partial_s\mu   \;\implies\;   \exists \: \partial_s(\Lift\mu) = \Lift(\partial_s\mu)$
	\end{enumerate}
\end{proposition}
\begin{proof}
	Let $M := \Lift(\mu)$.
	For any $A\in\A$, \: $F \in L(\V)$ and $x,h\in X$ we have:
	$$T\big( \Lift\mu, x \big)(A)[F] = M_x(A)[F] = M(A+x)[F] = \mu(A+x) \circ F = \mu_x(A) \circ F = \Lift \big[T( \mu, x )\big](A)[F]$$
	This gives the first assertion.
	Now take $h\in X$:
	\begin{multline*}
	\partial_h(\Lift\mu)(A)[F] = \left(\lim_{n\to\infty} \frac{M_{th}-M}{t}\right)(A)[F] =   \lim_{n\to\infty} \frac{M(A+th)[F]-M(A)[ F]}{t}  \\
	= \lim_{n\to\infty} \frac{\mu(A+th) \circ F-M(A) \circ F}{t} = \partial_h\mu(A) \circ F = \Lift\big[\partial_h\mu\big](A)[F]
	\end{multline*}
	The second assertion follows by induction.
\end{proof}

Now we pass to the definition of continuity ($k=0$) and smoothness.

\begin{definition}[Continuity and smoothness of a measure]
	Let $\Top \in \big\{\Top_{set}, \Top_{sv}, \Top_{ssv}, \Top_{v} \big\}$ and $k \geq 0$.
	\begin{enumerate}[label=\alph*)]
	\item
	Measure $\mu \in \M$ is called $k$ times continuously differentiable in topology $\Top$, if $T_\mu \in  C^k\big[X, (\M, \Top)\big]$. 
	The space of such measures is denoted by $C^k(Y, \Top)$.
	
	\item
	Analagously, measure $\mu \in \M$ is called $k$ times continuously differentiable along vector $h \in X$ in topology $\Top$, if $\forall x: [t \mapsto T_\mu(x+th)] \in C^k\big[\mathbb{R}, (\M, \Top)\big]$.
	The space of such measures is denoted by $C^k(Y, \Top, h)$.
	\end{enumerate}
\end{definition}

	When a measure $\mu$ belongs to the class $C^k, \: k \geq 2$ at least in topology $\Top_{set}$, by Schwartz Theorem \autocite[4.5.14]{Bogachev_Smolyanov_Sobolev}, all differentials of $T_\mu$ of order up to and including $k$ are symmetric multilinear operators.
	Consequently, the mixed Fomin derivatives do not depend on the order of differentiation, and we can speak of derivatives along multivectors instead of sequences of vectors.

\subsection{Results related to continuity}

\begin{theorem}[Continuity criterion in $\Top_{set}$] 
	Suppose $\mu \in \CA(\A,Y)$ \:and\: $(e_1 \cdots e_d)$ is a basis in $X$.\\
	Then the following conditions are equivalent:
	\begin{enumerate}[label=\alph*)]
	
	\begin{multicols}{3}
	\item
	$ \mu \in \bigcap_{i=1}^{d} C(Y, \Top_{set}, e_i)$
	
	\item
	$\displaystyle \mu << Leb_X$
	
	\item
	$\displaystyle \mu \in C(Y, \Top_{set})$
	
	\end{multicols}
	\end{enumerate}
\end{theorem}
\begin{proof}
	Obviously full continuity implies continuity along a basis.
	Suppose $\mu$ is continuous in $\Top_{set}$ along each basis vector.
	Consider a real-valued bounded linear functional $y^* \in Y^*_\mathbb{R}$.
	Then $y^*\mu$ is a real-valued countably additive measure, continuous along a basis.
	According to \autocite[Proposition 3.4.3]{Bogachev_DM} it is absolutely continuous with respect to Lebesgue measure.
	But then 
	$$\forall A \in \A:   \quad      Leb_X(A) = 0 \:\implies\: \|\mu(A)\| = \sup \big\{ |y^*\mu(A)| : y^* \in Y^*_{\mathbb{R}}, \: \|y^*\| \leq 1\big\} \:=\: 0$$
	Theorem \ref{7} yields $\mu << Leb_X$ and with Proposition \ref{36} we obtain $\tilde{\mu} << Leb_X$.

	Now assume $\tilde{\mu} << Leb_X$ and fix $\varepsilon > 0$.
	Since the Borel $\sigma$-algebra $\A$ on $\mathbb{R}^d$ is generated by the $\delta$-ring of bounded Borel sets, according to a statement in \autocite[p. 56]{Dinculeanu_VM},
	there exists a bounded Borel set $B \in \A$ such that $\tilde{\mu}(X) \leq \tilde{\mu}(B) + \varepsilon$.
	Then for any $A \in \A$ and $x\in X$:
	\begin{align*}
		\|T_\mu(x+h)[A] - T_\mu(x)[A]\|_Y \:&=\: \|\mu(A+x+h) - \mu(A+x)\|   \:\leq\:   \tilde{\mu}\Big( [A+x+h] \triangle [A+x] \Big) \\
		&\leq\:  \tilde{\mu}\Big( \big[[A+x+h] \triangle [A+x]\big] \cap B \Big) + \varepsilon
	\end{align*}
	Now we utilize the well-known fact that for $f \in L_1 \big( Leb_X, \mathbb{C}\big)$ the mapping $h \mapsto f(\cdot+h)$ is continuous in $L_1$-seminorm.
	Here the set $B$ compensates for the possibility that $\chi_A$ may not lie in $L_1$.
	\begin{align*}
		\tilde{\mu}\Big( \big[[A+x+h] &\triangle [A+x]\big] \cap B \Big)  \:<<\:    Leb_X \Big( \big[[A+x+h] \triangle [A+x]\big] \cap B \Big)  \:=\\
		&\int_X \big| \chi_{A+x}(r-h) - \chi_{A+x}(r)\big| \chi_B(r)\diff{r}  \:\leq \\
		&\int_X \big| \chi_{(A+x)\cap B}(r-h) - \chi_{(A+x)\cap B}(r) \big| \diff{r}   \:+\:   \int_X \big|  \chi_B(r-h) - \chi_B(r)\big| \diff{r} \:\xrightarrow{h\to0}\: 0
	\end{align*}
	This establishes continuity of $\mu$ in $\Top_{set}$.
\end{proof}

\begin{theorem}[Continuity criterion in $\Top_{sv}$] \label{37}
	Suppose $\mu \in \CA\big[\A, \: Y \hookrightarrow L(\V,\W) \big]$ and $(e_1 \cdots e_d)$ is a basis in $X$.\\
	Then the following conditions are equivalent:
	\begin{enumerate}[label=\alph*)]
	\item
	$\displaystyle \forall \: 1\leq i \leq d: \; \|\mu_{te_i} - \mu\|_{sv} \xrightarrow{t\to0} 0$ \quad (continuity of $T_\mu$ at zero along the basis)
	
	\item
	$\displaystyle \mu \in C(Y, \Top_{sv})$ \quad (continuity of $T_\mu$ everywhere)
	\end{enumerate}
\end{theorem}
\begin{proof}
	In one direction the implication is obvious.
	Suppose there is continuity at zero along the basis.
	Let $(e^1 \cdots e^d)$ be the dual basis. For a vector $h\in X$ we introduce the notation:
	$$h_1^{k} := \sum_{i=1}^k e^i(h)\cdot e_i  \:=\: \sum_{i=1}^k h^i \cdot e_i \quad\text{where}\quad 0 \leq k \leq d$$
	Take a disjoint family $\{A_i\}_1^m \in \mathrm{FD}(\A)$ and unit vectors $\{v_i\}_1^m$ from $\V$.
	Using translation invariance of total semivariation $\|\cdot\|_{sv}$ we obtain:
	\begin{multline*}
		\left\| \sum_{i=1}^{m} (\mu_h - \mu)[A_i]v_i \right\|        \:=\:       \left\| \sum_{i=1}^{m} \big[\mu(A_i + h) - \mu(A_i)\big]v_i \right\| \\
		=\: \left\| \sum_{i=1}^{m} \sum_{k=0}^{d-1} \big[\mu(A_i + h_{k+1}) - \mu(A_i + h_{k})\big]v_i \right\|     \;\leq\;   \sum_{k=0}^{d-1} \left\| \sum_{i=1}^{m} \big[ \mu(A_i + h_{k+1}) - \mu(A_i + h_{k})\big]v_i \right\| \\
		\leq\: \sum_{k=0}^{d-1} \big\| \mu_{h_{k+1}} - \mu_{h_k}\big\|_{sv}     \;=\;    \sum_{k=0}^{d-1} \big\| \mu_{h^{k+1} e_{k+1}} - \mu \big\|_{sv}
	\end{multline*}
	This yields continuity of $T_\mu$ at zero:
	$$ \big\| \mu_h - \mu \big\|_{sv}   \;\leq\;   \sum_{k=1}^{d} \big\| \mu_{h^{k} e_{k}} - \mu \big\|_{sv}   \;\xrightarrow{h\to0}\; 0$$
	Continuity at any other point follows again from translation invariance of total semivariation.
\end{proof}

\begin{theorem}[Continuity and operator-valued densities] \label{38}
	Suppose $\mu \in \CA\big[\A, \:L(\V,\W) \big]$ is regular with continuous semivariation $\widehat{\mu}$.\\
	Let $f \in \mathcal{L}_1(\mu, L(\V))$ \:and\: $\mathcal{H}$ be a countable filter base on  X, converging to zero.\\
	Then:\quad       $\displaystyle \big\| \mu_h - \mu \big\|_{sv} \xrightarrow{\mathcal{H}} 0    \;\implies\;       \big\| (\mu f)_h - \mu f \big\|_{sv} \xrightarrow{\mathcal{H}} 0$
\end{theorem}
\begin{proof}
	Recall that by $\mu f$ we mean $Mf$ where $M := \Lift(\mu)$. 
	Note that $\forall h\in X:\:\big\| \mu_h - \mu \big\|_{sv} = \big\| M_h - M \big\|_{sv}$ by \ref{30} and \ref{35}.
	Modifying $f$ on a $\widehat{\mu}$-null set if necessary, we may assume that $f \in \SM\big[\A,L(\V)\big]$.\\
	Suppose $\big\| \mu_h - \mu \big\|_{sv} \xrightarrow{\mathcal{H}} 0$.
	Fix $\varepsilon > 0$.
	By continuity of $\widehat{\mu}_1(f,\cdot) = \widehat{M}_1(f,\cdot)$ there exists $C>0$ such that $\displaystyle \widehat{\mu}_1\big(f, [|f|>C]\big) < \varepsilon$.
	Define $\displaystyle f_\varepsilon := \chi_{[|f|\leq C]} \cdot f$.
	Since $\widehat{\mu}$ is continuous, an analog of Lusin's Theorem \autocite[p. 519]{Dobrakov_1} is valid for functions in $\SM\big(\A, \V \big)$, so we have:
	$$\exists \: K \in \mathcal{K}_X : \quad  \widehat{\mu}(X\setminus K) < \varepsilon/C, \quad f_\varepsilon\restriction_K \in C\big[K, L(\V)\big]$$
	A generalization of the Tietze Extension Theorem \autocite[Theorem 4.1]{Dugundji_ETT} gives the existence of a continuous extension:
	$$ \exists \: g \in BC\big[ X, L(\V) \big] : \quad g\restriction_K = f_\varepsilon\restriction_K,   \quad   \big\|g\big\|_X \leq \big\|f_\varepsilon\big\|_K \leq C$$
	Clearly, $\im g$ is separable as a continuous image of a separable space, and by Theorem \ref{9} we conclude that $g \in \BSM\big[\A, L(\V)\big]$.
	In turn $\BSM\big[\A, L(\V)\big] \subseteq \mathcal{L}_1(M)$ by continuity of the semivariation $\widehat{M}$ and Theorem \ref{22}.
	As $g$ is continuous we have $g_h \to g$ pointwise along the filter base $\mathcal{H}$, and $|g_h| \leq \|g\|_X \in \mathcal{L}_1\big(\mu, [0,\infty]\big)$.
	Now, since $\mathcal{H}$ is countable, we can pass to the sequential characterization of convergence and invoke Theorem \ref{28} to conclude that $\widehat{M}_1\big( g_h - g ,  X \big) \xrightarrow{\mathcal{H}} 0$.
	Next, using translation invariance of total semivariation, we obtain:
	\begin{align*}
		&\big\| (\mu f)_h - \mu f \big\|_{sv}  :=  \big\| (Mf)_h - Mf \big\|_{sv} \leq \\
		&\big\| (Mf)_h - (Mf_\varepsilon)_h \big\|_{sv}      +      \big\| (Mf_\varepsilon)_h - (Mg)_h \big\|_{sv}         +       \big\| (Mg)_h - Mg_h \big\|_{sv}    + \\
		&+   \big\| Mg_h - Mg \big\|_{sv}    +   \big\| Mg - Mf_\varepsilon \big\|_{sv}      +      \big\| Mf_\varepsilon - Mf \big\|_{sv} = \\
		&2\big\| M(f - f_\varepsilon) \big\|_{sv}      +       2\big\| M(g - f_\varepsilon) \big\|_{sv}      +      \big\| M(g_h - g) \big\|_{sv}    +   \big\| (M_h - M) g_h \big\|_{sv} \leq \\
		&2 \widehat{M}_1\big( f , [|f|>C] \big)      \:+\:      2 \widehat{M}_1\big( g - f_\varepsilon , X\setminus K \big)      \:+\:       \widehat{M}_1\big( g_h - g ,  X \big)      \:+\:    C \big\| M_h - M \big\|_{sv} \:\leq \\
		&2\varepsilon + 4C\varepsilon/C +  \widehat{M}_1\big( g_h - g ,  X \big)      \:+\:    C \big\| M_h - M \big\|_{sv} \quad\xrightarrow{\mathcal{H}}\quad 6\varepsilon
	\end{align*}
	Since $\varepsilon$ is arbitrary, the proof is complete.
\end{proof}

\subsection{Results related to differentiability}
First we transfer several important facts known for real-valued measures to the case of a countably additive vector measure, working with its scalar semivariation (which is always continuous and finite).
\begin{theorem} \label{39}
	Suppose $\mu \in \mathcal{D}(Y, \Top_{set}, h)$. Then:
	\begin{enumerate}[label=\alph*)]
	\item
	$\displaystyle \forall A\in\A:\quad \big[\mathbb{R} \ni r \mapsto \mu(A+rh)\big]$ \; is Lipschitz continuous
	
	\item
	$\displaystyle \forall A\in\A:\quad \big[\mathbb{R} \ni r \mapsto \partial_h\mu(A+rh)\big]$ \; is locally Bochner integrable
	
	\item
	$\displaystyle \forall A\in\A:\quad \mu(A+th) - \mu(A) = \int_{0}^{t} \partial_h\mu(A+rh) \diff{r}$ \;(oriented integral)
	
	\item
	$\displaystyle \partial_h\mu \:<<\: \widetilde{\mu}$. \: If $\mu$ is regular, then so is $\partial_h\mu$
	\end{enumerate}
\end{theorem}
\begin{proof}
	The mapping $f: \mathbb{R} \ni r \mapsto \mu(A+rh)$ is everywhere differentiable with derivative $f' : \mathbb{R} \ni r \mapsto \partial_h\mu(A+rh)$.
	The first assertion is proved by applying the Mean Value Theorem for normed vector spaces \autocite[Theorem 7.2-1]{Ciarlet_LNFA}:
	$$\big\| \mu(A+th) - \mu(A+sh) \big\|_Y \:\leq\: |t-s| \sup\left\{  \big\|\partial_h(A+rh)\big\|_y   \::\: r \in \overleftrightarrow{(t,s)} \right\} \:\leq\: |t-s|\big\|\partial_h\mu\big\|_{ssv}$$
	(where the notation $\overleftrightarrow{(t,s)}$ stands for the line segement).
	In turn this implies that $f \in \SM(\mathcal{B}_\mathbb{R}, Y)$ and is absolutely continuous.
	Since $f'(r) = \lim_{n\to\infty} (f(r+1/n) - f(r))/(1/n)$, the derivative is also strongly measurable by Theorem \ref{9}.
	Local Bochner integrability is ensured by boundedness: $\|f'(r)\| \leq \|\partial_h\mu\|_{ssv} < \infty$.
	
	To prove the third statement take a functional $y^*\in Y^*$.
	Then $g:= y^*\circ f$ is also Lipschitz (hence absolutely continuous) and everywhere differentiable with $g' = y^*\circ f'$.
	Applying the Newton-Leibniz formula for the classical Lebesgue integral \autocite[Theorem 3.35]{Folland_RA} we get:
	$$y^*\big[ f(t) - f(0)\big] \:=\: g(t)- g(0) \:=\: \int_{0}^{t} g'(r) \diff{r} \:=\: y^*\left[\int_{0}^{t} f'(r) \diff{r}\right]$$
	where the integral is understood in the oriented sense.
	Since $Y^*$ separates points on $Y$, the third assertion is proved.
	
	Now we turn to proving $\partial_h\mu << \widetilde{\mu}$.
	Let $\widetilde{\mu}(A) = 0$ and take a real-valued functional $y^* \in Y^*_{\mathbb{R}}$.
	By Theorem \ref{5} we have $\overline{y^*\mu}(A) = 0$.
	From bounded linearity of $y^*$ it is obvious that $y^*\partial_h{\mu} = \partial_hy^*\mu$.
	Invoking a known result about real-valued measures differentiable along a vector \autocite[Corollary 3.3.2]{Bogachev_DM}, we obtain $\partial_hy^*\mu << \overline{y^*\mu}$.
	Consequently, $y^*\partial_h{\mu}(A) = y^*\big[\partial_h\mu(A)\big] = 0$.
	Since $Y^*_{\mathbb{R}}$ is norming for $Y$, we conclude that $\big\|\partial_h\mu(A)\big\|_Y = 0$.
	Absolute continuity follows from Theorem \ref{7}.
	
	Lastly, suppose $\mu$ is regular.
	Take $A\in\A$ and $\varepsilon>0$.
	By absolute continuity with respect to $\widetilde{\mu}$:
	$$ \exists \delta>0: \quad \widetilde{\mu}(B)<\delta \implies \overset{\smile}{\partial_h\mu}(B) < \varepsilon$$
	By definition of regularity:
	$$ \exists \: K \in \mathcal{K}_X,\: U \in \Top_X:  \quad  K\subseteq A \subseteq U \; , \; \overset{\smile}{\mu}(U\setminus K) < \delta$$
	Hence \: $\overset{\smile}{\partial_h\mu}(U\setminus K) < \varepsilon$ and we are done.
\end{proof}

Taking into consideration the isometry $Y \hookrightarrow L(\V,\W)$,  we can obtain further results interpreting the measure as operator-valued.
In such context we have two semivariations $\widetilde{\mu}$ and $\widehat{\mu}$.

\begin{theorem} \label{40}
	Suppose $\mu \in \mathcal{D}\big[ Y \hookrightarrow L(\V,\W), \Top_{set}, h \big]$. Then:
	\begin{enumerate}[leftmargin=*, label=\alph*)]
		\item
		$\displaystyle \forall \gamma \in \CA\big[\A, L(\V,\W)\big]: \; \svar_{\V,\W}\big[ \mu_{th} - \mu - t\gamma, E \big] \:\leq\: |t| \cdot \sup\left\{  \svar_{\V,\W}\big[ (\partial_h\mu)_{rh} - \gamma, E \big]  \::\:  r \in \overleftrightarrow{(0,t)} \right\}$
	
		\item \label{41}
		$\displaystyle \big\| (\partial_h\mu)_{th} - \partial_h\mu \big\|_{sv} \xrightarrow{t\to0} 0    \quad\implies\quad    \big\| (\mu_{th} - \mu)/t - \partial_h\mu \big\|_{sv} \xrightarrow{t\to0} 0$	
	\end{enumerate}
	If in addition to assumptions above $\widehat{\mu}$ is continuous on $\A$:
	\begin{enumerate}[resume, leftmargin=*, label=\alph*)]
	\begin{multicols}{3}
		\item
		$\widehat{\partial_h\mu}$ is continuous on $\A$ 

		\item
		$\displaystyle \big\| \mu_{th} - \mu \big\|_{sv} \xrightarrow{t\to0} 0$
	
		\item
		$\widehat{\partial_h\mu} \:<<\: \widehat{\mu}$
	\end{multicols}	
	\end{enumerate}
	If moreover $\mu$ is regular and $\partial_h\mu$ has an operator-valued Dobrakov $L_1$-density with respect to $\mu$:
	\begin{enumerate}[resume, leftmargin=*, label=\alph*)]
		\item
		$\displaystyle \big\| (\partial_h\mu)_{th} - \partial_h\mu \big\|_{sv} \xrightarrow{t\to0} 0$
	\end{enumerate}
\end{theorem}
\begin{proof}
	\vspace{-20pt}
	\begin{enumerate}[leftmargin=*, label=\alph*)]
		\item
		For $E\subseteq X$ take a disjoint family $\{A_i\}_1^m \in \mathrm{FD}(\A)$ contained in $E$ and some unit vectors $\{v_i\}_1^m$ from $\V$.
		Using the integral representation from \ref{39}, we get:
		\begin{align*}
			&\left\|  \sum_{i=1}^{m} (\mu_{th} - \mu - t\gamma)[A_i]v_i \right\|_\W \;=\; 
			\left\|  \sum_{i=1}^{m} \int_0^t  \Big[  \partial_h\mu(A_i + rh) v_i - \gamma(A_i) v_i \Big] \diff{r} \right\|_\W      \;= \\
			&\left\| \int_0^t \left[ \sum_{i=1}^{m}   \big(  (\partial_h\mu)_{rh} - \gamma \big) [A_i] v_i  \right] \diff{r} \right\|_\W   \;\leq\;   
			\int_0^t \left\|  \sum_{i=1}^{m}   \big(  (\partial_h\mu)_{rh} - \gamma \big) [A_i] v_i  \right\|_\W  \diff{r}    \;\leq \\
			&|t| \cdot \sup \left\{   \left\|  \sum_{i=1}^{m}   \big(  (\partial_h\mu)_{rh} - \gamma \big) [A_i] v_i  \right\|_\W  \::\:  r \in \overleftrightarrow{(0,t)}  \right\}    \;\leq \\
			&|t| \cdot \sup\left\{  \svar_{\V,\W}\big[ (\partial_h\mu)_{rh} - \gamma, E \big]  \::\:  r \in \overleftrightarrow{(0,t)} \right\}
		\end{align*}
		
		\item
		If $\displaystyle \big\| (\partial_h\mu)_{th} - \partial_h\mu \big\|_{sv} \xrightarrow{t\to0} 0$   \;\:then also\;\:   $\displaystyle \sup\left\{ \big\| (\partial_h\mu)_{rh} - \partial_h\mu \big\|_{sv}  \::\:  r \in \overleftrightarrow{(0,t)} \right\} \xrightarrow{t\to0} 0 $.\\
		Applying the first statement with $\gamma = \partial_h\mu$, we obtain:
		$$\big\| (\mu_{th} - \mu)/t - \partial_h\mu \big\|_{sv}      \;\leq\;     \sup\left\{ \big\| (\partial_h\mu)_{rh} - \partial_h\mu \big\|_{sv}  \::\:  r \in \overleftrightarrow{(0,t)} \right\}   \;\xrightarrow{t\to0} 0$$
		
		\item
		Since $\partial_h\mu$ is the setwise limit of the sequence $\big\{(\mu_{n^{-1}h} -\mu)/n^{-1}\big\}_{n=1}^\infty$, continuity of $\widehat{\partial_h\mu}$ follows from continuity of $\widehat{\mu}$ and \ref{6}.
		
		\item
		Since $\widehat{\partial_h\mu}$ is continuous, it is also finite, so $\widehat{\partial_h\mu}(X) < \infty$.
		Using the first assertion with $\gamma=0$ we get:
		$$\big\| \mu_{th} - \mu \big\|_{sv}   \;\leq\;   |t| \cdot \sup\left\{  \big\| (\partial_h\mu)_{rh} \big\|_{sv}  \::\:  r \in \overleftrightarrow{(0,t)} \right\}   \;\leq\;    |t| \cdot \widehat{\partial_h\mu}(X) \xrightarrow{t\to0} 0$$
		
		\item
		According to \autocite[*-Theorem, p. 515]{Dobrakov_1}, for a countably additive operator-valued measure on a $\sigma$-algebra continuity of semivariation implies absolute continuity of semivariation with respect to scalar semivariation.
		This fact together with Theorem \ref{39} gives \: $\widehat{\partial_h\mu} \:<<\: \widetilde{\partial_h\mu}  \:<<\: \widetilde{\mu}  \:\leq\: \widehat{\mu}$.
		
		\item
		The last assertion follows from Theorem \ref{38}.
	\end{enumerate}
\end{proof}

Given $\mu \in \mathcal{D}\big[ Y \hookrightarrow L(\V,\W), \Top_{set}, h \big]$, in accordance with the terminology adopted for the real-valued measures \autocite[Definition 3.3.8]{Bogachev_DM}, an operator-valued Dobrakov density (if it exists) of $\partial_h\mu$ with respect to $\mu$ will be called a logarithmic derivative of $\mu$ along $h$.

\begin{theorem}[Smoothness criterion in $\Top_{sv}$] \label{42}
	Suppose $\mu \in \CA\big[\A, \: Y \hookrightarrow L(\V,\W) \big]$ ,\: $k\in \mathbb{N}$ \:and\: $E = (e_1 \cdots e_d)$ is a basis in $X$.\\
	Let $S_0^kE$ be the set of finite sequences of basis vectors of length up to and including $k$.\\
	Then the following conditions are equivalent:
	\begin{enumerate}[label=\alph*)]
		\item \label{43}
		$\displaystyle \mu \in C^k\big[Y, \Top_{sv}\big]$
		
		\item \label{44}
		$\displaystyle \forall \: s \in S_0^kE, \: 1\leq i \leq d: \quad  \exists\; \partial_s\mu \quad\text{and}\quad  \; \big\|(\partial_s\mu)_{te_i} - \partial_s\mu\big\|_{sv} \xrightarrow{t\to0} 0$ 
	\end{enumerate}
\end{theorem}
\begin{proof}
	Obviously \ref{43} $\implies$ \ref{44}.
	We will prove the reverse implication by induction on $k$.
	Due to a well known fact from differential calculus in normed vector spaces \autocite[p. 85]{Zorich_2}, it suffices to show existence and continuity of all mixed basis derivatives of $T_\mu$ of order exactly $k$ in topology $\Top_{sv}$.
	
	Suppose $k=1$ and condition \ref{44} holds.
	By Theorem \ref{37} all the measures $\{\partial_i\mu \::\: 1\leq i \leq d\}$ are continuous in $\Top_{sv}$.
	Existence of partial derivatives along the basis is ensured by Theorem \ref{40}-\ref{41}:
	$$\left\| \frac{T_\mu(x+te_i) - T_\mu(x)}{t} - (\partial_i\mu)_{x} \right\|_{sv}    =     \left\| \frac{\mu_{x+te_i}-\mu_x}{t} - (\partial_i\mu)_x \right\|_{sv}    =     \left\| \frac{\mu_{te_i} - \mu}{t} - \partial_i\mu \right\|_{sv}   \xrightarrow{t\to0}0$$
	Their continuity is established as follows:
	$$\Big\| \partial_iT_\mu(x+h) - \partial_iT_\mu(x) \Big\|_{sv}     =      \Big\| (\partial_i\mu)_{x+h} - (\partial_i\mu)_{x} \Big\|_{sv}     =      \Big\| (\partial_i\mu)_{h} - \partial_i\mu \Big\|_{sv} \xrightarrow{t\to0} 0$$
	Thus $T_\mu \in C^1(X)$ and the claim is true for $k=1$.
	
	Now suppose the claim is true up to and including $k\geq1$, and condition \ref{44} holds for $k+1$.
	By the induction hypothesis $T_\mu \in C^k(X)$.
	Now we show that all its mixed basis derivatives of order $k+1$ in $\Top_{sv}$ exist and are continuous.
	Take $s \in S^kE$ (i.e. $|s|=k$).
	Then 
	$$\forall i,j       :\;        \big\| (\partial_i\partial_s\mu)_{te_j} - \partial_i\partial_s\mu \big\|_{sv}     =      \big\| (\partial_{si}\mu)_{te_j} - \partial_{si}\mu \big\|_{sv} \xrightarrow{t\to0} 0 $$
	and again by Theorem \ref{40}-\ref{41}:
	\begin{multline*}
		\left\| \frac{\partial_sT_\mu(x+te_i) - \partial_sT_\mu(x)}{t} - (\partial_{si}\mu)_{x} \right \|_{sv}      =
		\left\| \frac{(\partial_s\mu)_{x+te_i} - (\partial_s\mu)_{x}}{t} - (\partial_i\partial_s\mu)_{x}\right \|_{sv}    =\\
		\left\| \frac{(\partial_s\mu)_{te_i} - \partial_s\mu}{t} - \partial_i\partial_s\mu \right \|_{sv} \xrightarrow{t\to0} 0
	\end{multline*}
	Invoking once more Theorem \ref{37}, we obtain that all the measures $\{\partial_{si}\mu \::\: 1\leq i \leq d\}$ are continuous in $\Top_{sv}$.
	This translates into continuity of the partial derivatives:
		$$\left\| \partial_{si}T_\mu(x+h) - \partial_{si}T_\mu(x) \right \|_{sv}       =       \left\| (\partial_{si}\mu)_{x+h} - (\partial_{si}\mu)_{x} \right \|_{sv}       =      \left\| (\partial_{si}\mu)_{h} - (\partial_{si}\mu \right \|_{sv}   \xrightarrow{t\to0}  0 $$
\end{proof}

\begin{corollary}[A sufficient condition for smoothness] \label{45}
	Suppose $\mu \in \CA\big[\A, \: L(\V,\W) \big]$ is a regular measure with continuous semivariation $\widehat{\mu}$.
	If it has Dobrakov-$L_1$ logarithmic derivatives of order up to and including $k \geq 1$ along some basis, then $\mu \in C^k(\Top_{sv})$.
\end{corollary}

\subsection{Leibniz rule and the integration by parts formula} \label{46}

The first part of this subsection is concerned with an analog of Leibniz formula when differentiating the product of a measure and an operator-valued multiplier.
These results are proved under the assumptions of continuous semivariation and Fomin differentiability in topology $\Top_{sv}$.

\begin{theorem}[First order Leibniz formula in $\Top_{sv}$] \label{47}
	Suppose $\mu \in  \mathcal{D} \big[L(\V, \W), \Top_{sv}, h \big]$ with continuous semivariation $\widehat{\mu}$.\\
	Also assume \: $\varphi ,\: \partial_h\varphi \:\in\: \BSM\big[\A, L(\V)\big]$. \\
	Then: \qquad
	$\displaystyle  \mu \varphi \:\in\: \mathcal{D}\big[L(\V,\W), \Top_{sv}, h\big] \quad\text{and}\quad  \partial_h(\mu \varphi) \:=\: \partial_h\mu \cdot \varphi   \:+\:   \mu \cdot \partial_h\varphi$
\end{theorem}
\begin{proof}
	Again recall that multiplication between $\CA\big[\A, L(\V,\W)\big]$ and $\BSM\big[\A, L(\V)\big]$ is defined via the map $\Lift$.\\
	Let $M := \Lift(\mu)$.
	Continuity of $\widehat{\mu} = \widehat{M}$ ensures that $\BSM\big[\A, L(\V)\big] \subseteq \mathcal{L}_1\big[M, L(\V)\big]$.\\
	By the Mean Value Theorem \autocite[Theorem 7.2-1]{Ciarlet_LNFA} we have the following inequalities:
	$$\big\| (\varphi(x+th) - \varphi(x))/t \big\|    \;\leq\;     \sup \left\{ \big\|\partial_h\varphi(x+rh)\big\| \::\: r \in \overleftrightarrow{[0,t]} \right\}     \;\leq\;       \big\|\partial_h\varphi\big\|_X$$
	and consequently the Dominated Convergence Theorem \ref{28} gives $(\varphi_{th} - \varphi)/t \:\xrightarrow{t\to0}\: \partial_h\varphi$ in $L_1$ via a standard sequential argument.\\
	Next, for each $A\in\A$ we can write:
	$$\frac{(M \varphi)_{th} - M \varphi}{t}(A)    \;=\;     \int_{A} \frac{\varphi_{th} - \varphi}{t} \diff{[M_{th} - M]}     \;+\;    \int_{A} \frac{\varphi_{th} - \varphi}{t} \diff{M}       \;+\;     \int_{A} \varphi \diff{\left[ \frac{M_{th} - M}{t} \right]}$$
	Now we verify that this measure converges in semivariation to the specified limit as $t\to0$: 
	\begin{align*}
		&\left\| \frac{(M \varphi)_{th} - M \varphi}{t}    \:-\:    \partial_hM \: \varphi   \:-\:   M \: \partial_h\varphi \right\|_{sv}         \;\leq\;      \left\|  \int_{(\cdot)} \frac{\varphi_{th} - \varphi}{t} \diff{[M_{th} - M]} \right\|_{sv}          \;+\; \\
		&\left\|  \int_{(\cdot)} \frac{\varphi_{th} - \varphi}{t} \diff{M}    \:-\:   M \: \partial_h\varphi  \right\|_{sv}         \;+\;      \left\|  \int_{(\cdot)} \varphi \diff{\left[ \frac{M_{th} - M}{t} \right]}     \:-\:   \partial_hM \: \varphi   \right\|_{sv}         \;=\;      N_1(t)   +   N_2(t)    +   N_3(t) 
	\end{align*}
	The first summand tends to zero by semivariational continuity of $\mu$ along $h$:
	$$N_1(t)    \;\:\leq\;\:    \widehat{[M_{th}-M]}_1\left( \frac{\varphi_{th}-\varphi}{t} ,\: X \right)    \;\:=\;\:    \widehat{[\mu_{th}-\mu]}_1\left( \frac{\varphi_{th}-\varphi}{t} ,\: X \right)    \;\:\leq\;\:   \big\|\partial_h\varphi\big\|_X   \:  \big\| \mu_{th} -\mu \big\|_{sv}   \:\;\xrightarrow{t\to0}\:\; 0$$
	For the second we use the aforementioned $L_1$ convergence of the quotient:
	$$N_2(t)    \;\:\leq\;\:    \widehat{M}_1\left( \frac{\varphi_{th} -\varphi}{t} - \partial_h\varphi \:,\: X \right)    \;\:=\;\:    \widehat{\mu}_1\left( \frac{\varphi_{th} -\varphi}{t} - \partial_h\varphi \:,\: X \right)   \;\:=\;\:   \left\| \frac{\varphi_{th} -\varphi}{t} - \partial_h\varphi\right\|_{L_1}   \:\;\xrightarrow{t\to0}\:\; 0$$
	Lastly, for the third summand we utilize semivariational differentiability of the measure along $h$:
	$$N_3(t)    \;\:\leq\;\:   \widehat{\left[\frac{M_{th}-M}{t} - \partial_hM\right]}_1(\varphi,X)   \;\:=\;\:    \widehat{\left[\frac{\mu_{th}-\mu}{t} - \partial_h\mu\right]}_1(\varphi,X)    \;\:\leq\:\;     \big\| \varphi \big\|_X \left\| \frac{\mu_{th}-\mu}{t} - \partial_h\mu \right\|_{sv}   \:\;\xrightarrow{t\to0}\:\; 0$$
	The transition from $M$ back to $\mu$ in the lines above is justified by \ref{30} and \ref{35}.
\end{proof}

\begin{theorem}[Leibniz formula of order $k$ in $\Top_{sv}$] \label{48}
	Suppose $\mu \in  C^{k}\big[L(\V, \W), \Top_{sv} \big]$ with continuous semivariation $\widehat{\mu}$.
	Let $\varphi \in C_B^{k}\big[X, L(\V) \big]$.\\
	Then: \qquad
	$\displaystyle  \mu \varphi \:\in\: C^{k}\big[L(\V,\W), \Top_{sv}\big]     \qquad\text{and}\qquad     \forall h\in M_0^kX: \; \partial_h(\mu \varphi) = \sum_{\alpha+\beta=h} C(\alpha,\beta) \cdot \partial_{\alpha}\mu \cdot \partial_{\beta}\varphi  $
\end{theorem}
\begin{proof}
	The proof is carried out by induction.
	Take a basis $(e_1 \cdots e_d)$ for X.\\
	When $k=0$ the formula becomes an identity.
	Consider $k=1$.
	By Theorem \ref{47} the measure $\mu \varphi$ is differentiable in $\Top_{sv}$ along any vector and the Leibniz formula is valid.
	According to the continuity criterion \ref{42} it remains to verify continuity of basis Fomin derivatives at zero along the basis directions:
	\begin{align*}
		&\Big\| (\partial_i[\mu \varphi])_{te_j} - \partial_i[\mu \varphi] \Big\|_{sv}   \;=\; 
		\Big\|   (\partial_i\mu)_{te_j} \varphi_{te_j}   \:-\:  (\partial_i\mu)\varphi    \:+\:   \mu_{te_j} (\partial_i\varphi)_{te_j}   \:-\:   \mu(\partial_i\varphi)  \Big\|_{sv}  \;\leq\;  \\
		&\Big\|   \big[(\partial_i\mu)_{te_j}  - \partial_i\mu\big]\varphi_{te_j}   \Big\|       \:+\:        \Big\| \partial_i\mu \big[ \varphi_{te_j} -\varphi \big]   \Big\|         \:+\:        \Big\|  \big[ \mu_{te_j} -  \mu  \big]  (\partial_i\varphi)_{te_j}  \Big\|         \:+\:        \Big\|  \mu \big[  (\partial_i\varphi)_{te_j} -  \partial_i\varphi \big]    \Big\|      \leq\\
		&\big\| \varphi \big\|_X \Big\| (\partial_i\mu)_{te_j}  - \partial_i\mu \Big\|_{sv}    +\:    \widehat{\partial_i\mu}_1\Big[\varphi_{te_j} - \varphi , X\Big]      \:+\:      \big\|\partial_i\varphi \big\|_X \Big\| \mu_{te_j} -  \mu \Big\|_{sv}        +\:    \widehat{\mu}_1\Big[(\partial_i\varphi)_{te_j} -  \partial_i\varphi , X\Big] 
 	\end{align*}
	All summands in the last line converge to zero when $t\to0$ by the Dominated Convergence Theorem \ref{28} due to smoothness of $\mu$ and $\varphi$.
	Thus $\mu \varphi \:\in\: C^{1}\big[L(\V,\W), \Top_{sv}\big]$.\\
	Now assume the statement is true up to $k\geq 1$ inclusive, and the conditions are satisfied for $k+1$.
	Then $\mu \varphi \:\in\: C^{1}\big[L(\V,\W), \Top_{sv}\big]$ and the first order Leibniz formula holds for any vector.
	It suffices to show, that all the first order basis Fomin derivatives $\partial_i(\mu \varphi)    =    \partial_i\mu \: \varphi   +   \mu \: \partial_i\varphi$ are in the class $C^{k}\big[L(\V,\W), \Top_{sv}\big]$.
	But it follows from the induction hypothesis, since $\partial_i\mu, \: \varphi, \:   \mu, \: \partial_i\varphi$ are all of class $C^k$.
	Now for any multivector $h\in MX$ repeated application of the first order Leibniz formula with obvious combinatorial reasoning yields the expression in the statement of the theorem.
\end{proof}

Now we turn our attention to establishing the integration by parts formula.

\begin{theorem}[Integration by parts formula] \label{49}
	Suppose $\mu \in  \mathcal{D} \big[L(\V, \W), \Top_{set}, h  \big]$ with continuous semivariation $\widehat{\mu}$.\\
	Also let $f \in \BSM(\A,\V)$ \:and\: $1\leq p <\infty$. Then:
	\begin{enumerate}[label=\alph*)]
	\item
	$\displaystyle \forall \gamma \in \{\mu, \: \mu_{th}-\mu, \: \partial_h\mu\} \;:\; \BSM(\A,\V) \subseteq \mathcal{L}_p(\gamma, \V)$
	
	\item
	$\displaystyle \forall A \in \A    \;:\;     \int_{A} f \diff[\mu_{th}-\mu]     \:=\:    \int_{0}^{t} \int_{A+rh} \partial_h\mu(\diff x)  \: f(x-rh) \diff r$
	
	\item 
	If $f$ is continuous along $h$ almost everywhere $\widehat\mu$:
	$$\int_{X} \frac{\mu_{th}-\mu}{t}(\diff x)  \: f(x)  \;\xrightarrow{t\to 0}\;  \int_{X} \partial_h\mu(\diff x) \: f(x)          \qquad\qquad        \int_{X} \mu(\diff x) \: \frac{f_{th}-f}{t}(x)  \;\xrightarrow{t\to 0}\;  -\int_{X} \partial_h\mu(\diff x) \: f(x)$$
	
	\item
	If almost everywhere $\widehat\mu$ there exists a limit $\partial_hf = \lim_{t\to 0} (f_{th} - f)/t$ and eventually under $t\to 0$ it holds that $|(f_{th} - f)/t| \leq g \in \mathcal{L}_p(\mu)$, then $\partial_hf \in \mathcal{L}_p(\mu, \V)$ and:
	$$\int_{X} \mu(\diff x) \: \partial_hf(x)    \;=\;   - \int_{X} \partial_h\mu(\diff x) \: f(x)$$
	\end{enumerate}
\end{theorem}
\begin{proof}
	\vspace{-20pt}
	\begin{enumerate}[leftmargin=*, label=\alph*)]
		\item
		Continuity of $\widehat{\mu}$ implies continuity of $\widehat{\gamma}$ for $\gamma \in \{\mu, \: \mu_{th}-\mu, \: \partial_h\mu\}$, which in turn ensures $\displaystyle \BSM(\A,\V) \subseteq \mathcal{L}_p(\gamma, \V)$.
		
		\item
		Fix $t \in \mathbb{R}$.
		Due to the first statement we have $f \in \mathcal{L}_p(\mu_{th}-\mu, \V)$ and $f_{-rh} \in \mathcal{L}_p(\partial_h\mu, \V)$ for any $r\in\mathbb{R}$.\\
		First we prove the claim for simple functions.
		Let $\Sim(\A,\V) \ni s = \sum_{i=1}^{m} \chi_{B_i} v_i$ and $A\in\A$.
		From Theorem \ref{39} it follows that $ \left[r \mapsto \int_{A+rh} \partial_h\mu(\diff{x}) s(x-rh) \right]$ is locally Bochner integrable.
		Next we get the chain of equalities:
		\begin{multline*}
			\int_{A} s \diff{[\mu_{th} - \mu]}     \;=\;
			\sum_{i=1}^{m} \Big( \mu(A\cap B_i + th) - \mu(A\cap B_i) \Big)v_i     \;=\;\\
			\sum_{i=1}^{m} \left( \int_{0}^{t} \partial_h\mu(A\cap B_i + rh) \diff{r} \right)v_i    \;=\;
			\int_{0}^{t} \left( \sum_{i=1}^{m} \partial_h\mu(A\cap B_i + rh)v_i \right) \diff{r}   \;=\; \\
			\int_{0}^{t} \int_{A} (\partial_h\mu)_{rh}(\diff x) \cdot s(x) \diff{r} \;=\;
			\int_{0}^{t} \int_{A+rh} \partial_h\mu(\diff x) \cdot s(x-rh) \diff{r}
		\end{multline*}
		Invoking Theorem \ref{10}, we can take a sequence $\{s_n\}_1^\infty \in \Sim(\A,\V)^{\mathbb{N}}$ such that $s_n \to f$ and $|s_n| \nearrow |f|$.
		By the Dominated Convergence Theorem \ref{28} we obtain for any $A\in\A$ and $r\in\mathbb{R}$:
		\begin{align*}
			\int_{A} s_n \diff{[\mu_{th} - \mu]}     \;&\xrightarrow{n\to\infty}\;     \int_{A} f \diff{[\mu_{th} - \mu]} \\
			g_n(r,A)    \;:=\;     \int_{A+rh} \partial_h\mu(\diff x) \cdot s_n(x-rh) \diff{r}         \;&\xrightarrow{n\to\infty}\;        \int_{A+rh} \partial_h\mu(\diff x) \cdot f(x-rh) \diff{r}     \;=:\;     g(r,A)
		\end{align*}
		As already mentioned, $g_n(\cdot, A)$ are Bochner integrable on $\overleftrightarrow{[0,t]}$ and it is obvious form the definition that $\big\|g_n(r,A)\big\|_{\W} \leq \big\|f\big\|_X \cdot \widehat{\partial_h \mu}(X)$.
		Once again applying the Dominated Convergence Theorem, this time for the Bochner integral \autocite[Theorem VI.5.8]{Lang_RFA}, we obtain $g(\cdot, A) \in L_1\left(\overleftrightarrow{[0,t]}, \W \right)$ and :
		$$ \int_{0}^{t} g_n(r,A) \diff{r}     \;\xrightarrow{n\to\infty}\;      \int_{0}^{t} g(r,A) \diff{r}$$
		Combining the above results gives:
		$$\int_{A} s_n \diff{[\mu_{th} - \mu]}       \;=\;        \int_{0}^{t} \int_{A+rh} \partial_h\mu(\diff x) \cdot s_n(x-rh) \diff{r}     \;\xrightarrow{n\to\infty}\;      \int_{0}^{t} \int_{A+rh} \partial_h\mu(\diff x) \cdot f(x-rh) \diff{r}$$
		
		\item
		As a special case of the previous statement we have:
		$$\int_{X} f \diff\left[\frac{\mu_{th}-\mu}{t} \right]    \;=\;   \frac{1}{t}   \int_{0}^{t} \int_{X} \partial_h\mu(\diff x)  \: f(x-rh) \diff{r}$$
		Suppose $f$ is continuous along $h$ almost everywhere $\widehat{\mu}$.\\
		Note that $\forall N: \; \widehat{\mu}(N)=0 \implies \widetilde{\partial_h\mu}(N)=0  \implies \widehat{\partial_h\mu}(N)=0$.\\
		Therefore $\widehat{\partial_h\mu}$ - almost everywhere we have $\big\| f(x+th) - f(x) \big\| \xrightarrow{t\to0} 0$.\\
		Theorem \ref{28}  now implies that \; $\widehat{\partial_h\mu}_1\big( f_{th} - f , X\big) \xrightarrow{t\to0} 0$.
		This gives the first limit as follows:
		\begin{multline*}
			\left\|  \int_{X} f \diff\left[\frac{\mu_{th}-\mu}{t} \right]    \quad-\quad   \int_{X} f \diff{\big[\partial_h\mu\big]}  \right\|_{\W} = \\
			\left\|  t^{-1} \int_{0}^{t} \int_{X} \partial_h\mu(\diff x)\:f(x-rh) \diff{r}      \quad-\quad    \int_{X} \partial_h\mu(\diff x)\:f(x)    \right\|_{\W} =\\
			t^{-1} \left\|  \int_{0}^{t} \int_{X} \partial_h\mu(\diff x)\:  \big[ f(x-rh) - f(x) \big] \diff{r}     \right\|_{\W}  \leq \\
			t^{-1}  \int_{\overleftrightarrow{[0,t]}}   \left\|\int_{X} \partial_h\mu(\diff x)\:  \big[ f(x-rh) - f(x) \big]  \right\|_{\W}  \diff{r} \leq \\
			\sup \left\{   \left\|\int_{X} \partial_h\mu(\diff x)\:  \big[ f(x-rh) - f(x) \big]  \right\|_{\W}    \;:\; r \in \overleftrightarrow{[0,t]} \right\} \leq \\
			\sup \left\{    \widehat{\partial_h\mu}_1\big( f_{-rh} - f , X\big)    \;:\; r \in \overleftrightarrow{[0,t]} \right\}    \;\xrightarrow{t\to0} 0
		\end{multline*}
		The second limit is derived from the first by transforming the integral:
			$$\int_{X}  \frac{f_{th}-f}{t} \diff\mu     \;=\;     t^{-1} \left( \int_{X} f \diff{\mu_{-th}}   \;-\;  \int_{X} f \diff{\mu} \right)    \;=\;  - \int_{X} f \diff\left[\frac{\mu_{-th}-\mu}{-t} \right]    \;\xrightarrow{t\to0}\;      -\int_{X} f \diff{\big[\partial_h\mu\big]}   $$
		
		\item
		If the conditions of this clause are satisfied, Theorem \ref{28} gives $\partial_hf \:\in\: \mathcal{L}_p(\mu,\V)$ and:
		$$\int_{X} \partial_hf \diff{\mu}    \;=\;   \lim_{t\to0} \int_{X} \frac{f_{th} - f}{t} \diff{\mu}    \;=\;   - \int_{X} f \diff{\big[\partial_h\mu\big]}  $$
	\end{enumerate}
\end{proof}

Note that, although Theorem \ref{49} is stated for $\V$-valued functions, the fact that the operator $\Lift$ is a linear isometry commuting with translations and derivatives, allows to elevate this result verbatim to $L(\V)$-valued functions.

As a direct corollary of Theorems  \ref{49} and \ref{48} we immediately obtain the integration by parts formula of higher order.

\begin{corollary}[Integration by parts formula of order $k\geq0$] \label{50}
	Suppose $\mu \in  C^{k}\big[L(\V, \W), \Top_{sv} \big]$ with continuous semivariation $\widehat{\mu}$.\\
	Also assume $f \in C_B^{k}\big[X, \V \big]$    \:and\:     $\varphi \in C_B^{k}\big[X, L(\V) \big]$.\\
	Then for any multivector $h \in M_0^kX$ of order up to and including $k$ the following formula holds:
	$$\int_{X} \mu(\diff x) \: \varphi(x) \: \partial_hf(x)        \;\:=\;\:       (-1)^{|h|}  \sum_{\alpha+\beta=h} C(\alpha,\beta) \int_X \partial_{\alpha}\mu(\diff{x}) \: \partial_{\beta}\varphi(x) \: f(x)$$
\end{corollary}


\newpage
\section{Vector-valued distributions w.r.t. an operator-valued measure} \label{57}
As noted in \autocite{Bogachev_DM}, a general approach to distribution theory was proposed by S.V. Fomin as a way to extend the finite-dimensional theory of L. Schwartz to infinite-dimensional locally convex spaces.
An exposition of this approach can be found in \autocite{Daletsky_Fomin}.
It involves the introduction of four classes of objects: the space of test functions, the space of test measures, dual to the first, the space of generalized measures, and, dual to the second, the space of generalized functions.
The transport of operations to the generalized spaces is carried out by means of appropriate dual maps.
In the classical case of scalar-valued functions on a finite-dimensional space the presence of the canonical Lebesgue measure allows to identify the two test spaces via densities.

In our case the aim is to replace the Lebesgue measure with an operator-valued measure $\mu$, satisfying some appropriate requirements. 
Despite the finite dimensionality of the domain, the aforementioned identification is not possible in this situation, due to the fact that test functions are vector-valued, whereas densities are in general operator-valued.
Furthermore, the goal of this section is to rigorously arrive at such a definition of generalized differentiation that will match the idea of Sobolev derivative described in an ad hoc fashion in \ref{0}.
The general scheme from \autocite{Daletsky_Fomin} will be used to achieve this; moreover, not one but two spaces of test measures will be required.

\subsection{Test and generalized spaces}

Throughout this section, as in the previous one, we have $(X,\A):= (\mathbb{R}^d, \mathcal{B}_{\mathbb{R}^d})$.
As before $\V, \W$ are Banach spaces over $\F \in \{\mathbb{R}, \mathbb{C}\}$.
We also fix the exponent $1 < p < \infty$ together with its conjugate $p^{-1}+q^{-1}=1$.

From this point onwards we will work with a fixed measure $\mu \:\in\: \CA\big[\A, L(\V,\W)\big]$, satisfying the following conditions: continuity of semivariation $\widehat{\mu}$,  regularity and full support $\supp{\mu} = X$.
The later means that the semivariation $\widehat{\mu}$ does not vanish on any nonempty open set.
We also assume that $\mu$ has logarithmic derivatives $\Theta_h, \: h \in M_0^kX$ of order up to and including $k\geq1$ and all of them belong to $\mathcal{L}_q\big[ \mu, L(\V)\big]$.
From Corollary \ref{45} we know that these conditions imply $\mu \in C^k\big[ L(\V,\W), \Top_{sv}\big]$.

\begin{definition}[Space of test functions]
	The space of test functions is defined as:\\
	$\TF := C_B^\infty(X,\V)$ (smooth functions with bounded derivatives of all orders).\\
	It is equipped with the family of seminorms $\displaystyle \{\rho_j\}_{j\geq0}$ where $\displaystyle \rho_j(f) \;:=\; \big\|f^{(j)}\big\|_X \;=\;  \sup_{x\in X} \big\| f^{(j)}(x) \big\|_{L_j(X,\V)} $
\end{definition}

\begin{definition}[Spaces of test measures]
	The spaces of test measures are defined as follows:
	\begin{enumerate}[label=\alph*)]
	\item
	$\displaystyle \TM_B := \Big\{ \mu \varphi\::\: \varphi \in C_B^\infty\big[X,  L(\V) \big] \Big\}$. \\
	It is equipped with the family of seminorms $\displaystyle \{\rho_j\}_{j\geq0}$ where $\displaystyle \rho_j(\mu \varphi) \;:=\; \big\|\varphi^{(j)}\big\|_X \;=\;  \sup_{x\in X} \big\| \varphi^{(j)}(x) \big\|_{L_j(X,L(\V))} $
	
	\item
	$\displaystyle \TM_q := \Big\{ \mu g \::\: g \in \mathcal{L}_q\big[\mu, L(\V)\big] \Big\}$.\\
	It is equipped with the norm $\| \mu g\|_q := \| g\|_q = \widehat{\mu}_q(g,X)$ 
	\end{enumerate}
\end{definition}
In the above definition the seminorm of a measure is determined by the seminorm of its multiplier, which in both cases lies in $\mathcal{L}_1\big[\A, L(\V) \big]$.
Proposition \ref{27} assures that such a measure can vanish only if the multiplier is zero $\widehat{\mu}$ almost everywhere.
For a smooth multiplier this would mean everywhere, since $\mu$ is assumed to have full support.
Therefore in both cases the space of measures is in bijective correspondence with the space of multipliers, and the topologies are correctly determined. 
Obviously $\TM_B \subseteq \TM_q$.
 
Naturally, there arises the following bilinear pairing operation, continuous in the chosen topologies:
$$\TM_q \times \TF    \;\ni\;   (\mu g , f)   \;\mapsto\;   \pairing{\mu g}{f}  \;:=\;   \int_X \mu g(\diff{x})\:f(x)    \;=\;    \int_X \mu(\diff{x})\:g(x)\big[f(x)\big]    \;\in\;   \W$$
$$\big\| \pairing{\mu g}{f} \big\|_\W    \;\leq\;   \widehat{\mu}_1\big[gf, X\big]   \;\leq\;         \big\| f \big\|_X  \:    \big\| \mu \big\|_{sv}^{1/p}    \:     \big\| g \big\|_{q} $$

\begin{definition}[Spaces of generalized measures and functions]
	We define the following generalized spaces:
	\begin{enumerate}[label=\alph*)]
		\item
		$\W$-valued generalized measures : \; $\TF'(\W)    \;:=\;   LC\big( \TF , \W \big)$
		
		\item
		$\W$-valued generalized functions : \; $\TM_B'(\W)    \;:=\;   LC\big( \TM_B , \W \big)    \qquad     \TM_q'(\W)    \;:=\;   LC\big( \TM_q , \W \big)$		
	\end{enumerate}	
		Continuity of linear operators is understood in terms of the topologies introduced above.\\
		All three spaces are equipped with the topology of pointwise convergence  (strong operator topology).\\
		Pairing between test and generalized objects will also be denoted by $\pairing{\cdot}{\cdot}$.
\end{definition}

\begin{lemma} \label{51}
	For all $1 \leq r < \infty$ the sets $C_C^\infty(X,\V)$ and $C_C^\infty\big[X,L(\V)\big]$ are dense in $\mathcal{L}_r(\mu, \V)$ and $\mathcal{L}_r\big[\mu, L(\V)\big]$ respectively.
\end{lemma}
\begin{proof}
	Note that, as far as the measure is concerned, only the assumptions of continuous semivariation and regularity are needed for this lemma.
	We first treat the space $C_C^\infty(X,\V)$.
	It follows from Theorems \ref{10} and \ref{28} that $\Sim(\A,\V)$ is dense in $\mathcal{L}_r(\mu, \V)$.
	Hence it suffices to show that $C_C^\infty(X,\V)$ approximates any function of the form $\chi_A v$ where $v\in \V, A\in \A$.
	Fix $\varepsilon > 0$.
	Recalling Proposition \ref{36} and \autocite[*-Theorem, p. 515]{Dobrakov_1}, we have $\widehat{\mu} << \widetilde{\mu} \leq 4 \overset{\smile}{\mu}$.
	So we may replace $\overset{\smile}{\mu}$ with $\widehat{\mu}$ in the definition of regularity and approximate $A$ in semivariation with compact and open sets:
	$$\exists \: K \in \mathcal{K}_X, \: U \in \Top_X: \quad K \subseteq A \subseteq U, \quad \widehat{\mu}(U \setminus K) < \varepsilon$$
	Next, we can find a smooth bump function \autocite[Proposition 2.25]{Lee_ISM}:
	$$\varphi \in C_C^\infty\big( X, [0,1] \big) : \quad \varphi\restriction K = 1, \quad \supp{\varphi} \subseteq U$$
	Clearly $\varphi v \in C_C^\infty(X,\V)$ and there is a chain of inequalities:
	$$\widehat{\mu}_r\big[ \chi_A v - \varphi v, X\big]   \;\leq\;   \|v\| \: \widehat{\mu}_1\big[ |\chi_A  - \varphi |^{r} , X\big]^{1/r}    \;\leq\;   \|v\| \: \widehat{\mu}(U \setminus K)^{1/r}   \;\leq\;   \|v\| \: \varepsilon^{1/r}$$
	This proves the claim.
	The second assertion reduces to the first by passing to the measure $\Lift(\mu)$, which by Proposition \ref{30} is also regular with continuous semivariation.
\end{proof}

\begin{theorem} \label{52}
	Thre exist continuous linear injections: \;
	$\TF     \:\hookrightarrow\:       L_p(\mu, \V)    \:\hookrightarrow\:    \TM'_q(\W)    \:\hookrightarrow\:    \TM'_B(\W)$
\end{theorem}
\begin{proof}
	We begin with the first injection.
	It is obvious that $\TF \subseteq \BSM(\A,\V) \subseteq \mathcal{L}_p(\mu,\V)$.
	Continuity of the inclusion is easily seen from:
	$$\forall f \in \TF: \; \big\|f\big\|_p = \widehat{\mu}_1(|f|^p,X)^{1/p} \leq \big\|f\big\|_X \widehat{\mu}(X)^{1/p}$$
	We need to show its injectivity as a mapping into $L_p(\mu,\V)$.
	If a function $f \in \TF$ is nonzero at least at one point, by coninuity it must be different from zero on some open set $U\subseteq X$.
	Since $\mu$ has full support, we get $\widehat{\mu}(U) > 0$.
	Thus, $f$ lies in a nonzero equivalence class of $L_p$.
	
	Now we turn to establishing the second injection as the following mapping:
	$$L_p(\mu, \V) \;\ni\; f \;\mapsto\; \pairing{\cdot}{f} \;\in\;  \TM'_q(\W)$$
	Its image is indeed contained in $\TM'_q(\W)$ due to H\"older's inequality \ref{23}.
	To verify its injectivity, let $0 \neq f \in L_p(\mu,\V)$.
	Proposition \ref{27} insures that the integral measure $\mu f$ is not identically zero.
	Therefore, there exists $A \in \A$ such that $\int_A f \diff{\mu} \neq 0$.
	But this means $\pairing{\mu g}{f} \neq 0$, where $g = \chi_A \operatorname{Id}_\V \in \mathcal{L}_q\big[\mu, L(\V)\big]$ due to continuity of $\widehat{\mu}$.
	Next, to show continuity of the injection, let $\mathcal{L}_p(\mu,\V) \ni f_n \xrightarrow{L_p} f \in \mathcal{L}_p(\mu,\V)$.
	For any $g \in \mathcal{L}_q\big[\mu,L(\V)\big]$ we have:
	\begin{multline*}
		\big\| \pairing{\mu g}{f} - \pairing{\mu g}{f_n} \big\|_\W   \;=\;  \left\| \int_X g[f - f_n] \diff{\mu} \right\|_\W    \;\leq\;   \widehat{\mu}_1\big( g[f-f_n], X \big) \\
		\;\leq\;   \widehat{\mu}_1\big( |g||f-f_n|, X \big)  \;\leq\;  \widehat{\mu}_p\big( f-f_n, X \big) \: \widehat{\mu}_q\big( g , X \big)  \;\xrightarrow{n\to\infty}\; 0
	\end{multline*}
	Therefore $f_n \to f$ in $\TM_q'(\W)$ as generalized functions.
	
	The third mapping is given by the restriction to the space $\TM_B$:
	$$\TM'_q(\W)    \;\ni\;   F   \;\mapsto\;    F\restriction_{\TM_B}   \;\in\;       \TM'_B(\W)$$
	The fact that its image lies in  $\TM'_B(\W)$ follows from the inequalities:
	$$\forall \varphi \in C_B^\infty\big[X, L(\V)\big]: \;  \big\| \pairing{\mu \varphi}{F} \big\|_\W    \;\leq\;    \big\|F\big\|_{op} \big\| \varphi \big\|_q  \;\leq\;  \big\|F\big\|_{op} \big\| \varphi \big\|_X  \big\| \mu \big\|_{sv}^{1/q}$$
	Its injectivity is a consequence of Lemma \ref{51}: 
	if $F\restriction \TM_B = 0$ then by continuity of $F$ and density of $C_B^\infty\big[X,L(\V)\big]$ in $\mathcal{L}_q\big[\mu, L(\V)\big]$ we obtain $F=0$.
	Continuity of the injection is obvious: 
	if a net $\big\langle F_t \big\rangle$ converges pointwise to $F \in \TM'_q(\W)$, then so do the restrictions.
\end{proof}

\begin{theorem}
	There exist continuous linear injections: \;
	$\TM_B     \:\hookrightarrow\:      \TM_q      \:\hookrightarrow\:     \TF'(\W)$
\end{theorem}
\begin{proof}
Continuity of the inclusion $\TM_B  \subseteq  \TM_q$ is easily seen from:
$$ \forall \varphi \in C_B^\infty\big[X, L(\V)\big]:      \;         \big\| \varphi \mu \big\|_q      \;:=\;       \big\| \varphi \big\|_q      \;\leq\;     \big\| \varphi \big\|_X \big\| \mu \big\|_{sv}^{1/q}$$
The second mappinf is given by:
$$\TM_q \;\ni\; \gamma \;\mapsto\; \pairing{\gamma}{\cdot} \;\in\;  \TF'(\W)$$
To show its injectivity, suppose $\gamma = \mu g \in \TM_q$ and $\pairing{\gamma}{\cdot} = 0$, i.e. $f \in \TF: \: \int_X f \diff{\gamma} = 0$.
Take arbitrary $ A\in\A, \: v\in \V$.
By Lemma \ref{51} there exists a sequence $\TF \ni f_n \xrightarrow{L_p} \chi_A v$.
Then: 
\begin{multline*}
	\big\| \gamma(A) v \big\|_{\W}    \;=\;    \left\| \int_{X} \chi_A v \diff{\gamma} \right\|_{\W}    \;\leq\;     \left\| \int_{X} [\chi_A v - f_n]  \diff{\gamma} \right\|_{\W}     +    \left\| \int_{X} f_n \diff{\gamma} \right\|_{\W}    \\
	=\;    \left\| \int_{X} g[\chi_A v - f_n]  \diff{\mu} \right\|_{\W}    \;\leq\;    \widehat{\mu}_1\big( g[\chi_A v - f_n] , X\big)    \;\leq\;     \big\|g\big\|_q  \big\| \chi_A v - f_n \big\|_p   \;\xrightarrow{n\to\infty}\;   0
\end{multline*}
From here we conclude that $\forall A \in \A: \gamma(A) = 0$.\\
Now we verify continuity of this injection.
Let $\TM_q \ni \gamma_n = \mu g_n \xrightarrow{n\to\infty} \gamma \in \TM_q$.
Then for any $f \in \TF$ we have:
$$\big\| \pairing{\gamma_n}{f} - \pairing{\gamma}{f} \big\|_\W     \;=\;    \left\| \int_X f \diff{[\gamma_n - \gamma]} \right\|_\W    \;=\;      \left\| \int_X [g_n-g]f \diff{\mu} \right\|_\W    \;\leq\;  \big\|f\big\|_p \big\| g_n - g \big\|_q    \:\xrightarrow{n\to\infty}\:0$$
Thus, $\gamma_n \to \gamma$ in $\TF'$ as generalized measures.
\end{proof}


\subsection{Generalized derivatives}

In light of Theorem \ref{52} let us fix the notation for the two continuous linear injections, each of which maps an element of $L_p(\mu,\V)$ to the induced regular generalized function on the respective space of test measures:
$$\mathcal{R}_q \::\: L_p(\mu,\V)  \:\ni\:   f   \:\mapsto\:     \pairing{\cdot}{f}     \:\in\:      \TM'_q(\W)              \qquad             \mathcal{R}_B    \::\:  L_p(\mu,\V)    \:\ni\:   f   \:\mapsto\:    \pairing{\cdot}{f}\restriction \TM_B   \:\in\:   \TM'_B(\W)$$
Let us also recall that for any continuous linear operator $\mathcal{T} \in LC\big(\TM_B , \TM_q \big)$ there is a continuous transpose:
$$ \mathcal{T}'  \::\:   \TM_q'(\W) \:\ni\: F  \:\mapsto\:   F \circ \mathcal{T}   \:\in\:   \TM_B'(\W)$$

\begin{definition}[The general scheme for transporting operations onto the space of generalized functions] \label{54}
	Suppose we have the following initial data:\\
	$\TF_0  \:\subseteq\:  \TF$ \;- some linear subspace\\
	$\mathcal{T}_0 \::\: \TF_0 \:\to\: \TF$ \;- a linear operator (not necessarily continuous)\\
	$\mathcal{T}_1 \::\: \TM_B \:\to\: \TM_q$ \;- a continuous linear operator\\
	Then the statements below are equivalent formulations of the same condition:
	\begin{enumerate}[label=\alph*)]
		\item
		$\displaystyle \forall \gamma \in \TM_B, \: f \in \TF_0 : \; \pairing{\gamma}{\mathcal{T}_0 f} = \pairing{\mathcal{T}_1 \gamma}{f}$
		
		\item
		$\displaystyle \mathcal{R}_B \circ \mathcal{T}_0 = \mathcal{T}'_1 \circ \mathcal{R}_q \restriction \TF_0$
	\end{enumerate}
	If this condition is satisfied, $\mathcal{T}'_1$ can be considered a continuous linear extension of $\mathcal{T}_0$ to the space $\TM'_q(\W)$, bearing in mind the injection $\TF_0 \hookrightarrow \TM'_q(\W)$.
	The result of this extended operation lies in $\TM'_B(\W)$.
\end{definition}

In particular, the described scheme allows to introduce the operation of generalized differentiation on the space $\TM'_q(\W)$ in a way consistent with classical differentiation on $\TF$.
The construction involves transposing the Fomin differentiation operator, and the adjointness condition is ensured by the integration by parts formula, derived in the preceding section.
First it must be verified that Fomin differentiation is continuous in the chosen topologies.

\begin{theorem}[Continuity of Fomin differentiation] \label{53}
	If $h \in M_0^kX$ is a multivector of order not exceeding $k$, then $\displaystyle \partial_h \:\in\: LC\big[  \TM_B , \TM_q \big]$.
\end{theorem}
\begin{proof}
	From our assumptions on $\mu$ and Corollary \ref{48} we know that $\TM_B \subseteq C^k\big[L(\V,\W), \Top_{sv}\big]$,	
	hence the operation $\partial_h : \TM_B \to \CA\big[ \A, L(\V,\W)\big]$ is well-defined and is obviously linear.\\
	First we check that its image lies in $\TM_q$.
	Let $\varphi \in C_B^\infty\big[  X, L(\V) \big]$.
	Applying the Leibniz formula \ref{48} and appealing to the existence of logarithmic derivatives in $\mathcal{L}_q\big[\mu, L(\V)\big]$, we obtain:
	$$\partial_h(\mu \varphi)     \;=\;       \sum_{\alpha+\beta=h} C(\alpha,\beta) \cdot \partial_{\alpha}\mu \cdot \partial_{\beta}\varphi    \;=\;   \sum_{\alpha+\beta=h} C(\alpha,\beta) \cdot \mu \cdot \Theta_{\alpha} \cdot \partial_{\beta}\varphi$$
	$$\widehat{\mu}_q\big( \Theta_{\alpha} \cdot \partial_{\beta}\varphi  \:,\: \cdot \big)    \;\leq\;     \big\| \partial_{\beta}\varphi \big\|_X \: \widehat{\mu}_q\big( \Theta_{\alpha}, \cdot \big)$$
	Consequently, all $\Theta_{\alpha} \cdot \partial_{\beta}\varphi$ lie in $\mathcal{L}_q\big[\mu, L(\V)\big]$, and $\partial_h(\mu \varphi) \in \TM_q$. \\
	Continuity can be verified as follows:
	\begin{align*}
		\big\|  \partial_h(\mu \varphi) \big\|_q     \;&=\;     \widehat{\mu}_q \left[   \sum_{\alpha+\beta=h} C(\alpha,\beta) \cdot  \Theta_{\alpha} \cdot \partial_{\beta}\varphi    \:,\;  X \right]         \;\leq\;      \sum_{\alpha+\beta=h} C(\alpha,\beta)  \cdot \widehat{\mu}_q \big[  | \Theta_{\alpha} |  | \partial_{\beta}\varphi |  \:,\:  X \big]   \\
		&\leq\:      \max_{0 \leq j \leq |h|} \max_{x \in X} \big\|\varphi^{(j)}(x) \big\|_{L_j(X,L(\V))}   \cdot    \left(\prod_{x \in X} \|x\|^{h(x)}\right)   \cdot    \sum_{\alpha+\beta=h} C(\alpha,\beta)  \big\| \Theta_{\alpha} \big\|_q   \\
		&=\;     \mathrm{const}(h)       \cdot        \max_{0 \leq j \leq |h|} \rho_j(\mu \varphi)   
	\end{align*}
\end{proof}

Now we can layout the precise setup in line with Definition \ref{54}, verify the adjointness condition and obtain the explicit expression for the derivative.

\begin{lemma}
	Let $h \in M_0^kX$ be a multivector of order not exceeding $k$. Define the following operations:\\
	$\displaystyle \mathcal{T}_0 \::\: \TF \:\ni\:  f  \:\mapsto\:  \partial_hf \:\in\: \TF$ \; (continuous linear operator)\\
	$\displaystyle \mathcal{T}_1 \::\:  \TM_B   \:\ni\:  \gamma  \:\mapsto\:  (-1)^{|h|} \partial_h\gamma  \:\in\: \TM_q$ \; (continuous linear operator by \ref{53})\\
	Then:
	\begin{enumerate}[label=\alph*)]
	\item
	$\displaystyle \forall \gamma \in \TM_B, \: f \in \TF : \quad         \pairing{\gamma}{\mathcal{T}_0 f} = \pairing{\mathcal{T}_1 \gamma}{f}$
	
	\item
	$\displaystyle \forall \varphi \in C_B^\infty\big[X, L(\V)\big], \:F \in \TM'_q(\W) : \quad        \pairing{\mu\varphi}{\mathcal{T}'_1 F}    \;=\;     (-1)^{|h|}  \sum_{\alpha+\beta=h} C(\alpha,\beta)  \pairing{\mu \: \Theta_{\alpha} \: \partial_{\beta}\varphi}{F}$	
	\end{enumerate}
\end{lemma}
\begin{proof}
	For the empty multivector $\mathcal{T}_0$ and $\mathcal{T}_1$ become the identities and both claims are trivially valid.\\
	Assume $0 < |h| \leq k$.
	To show the first claim, take $f\in \TF, \: \varphi \in C_B^\infty\big[X, L(\V)\big]$ and apply the integration by parts formula \ref{49} $|h|$ times:
	$$\pairing{\mu \varphi}{\mathcal{T}_0 f}     \;=\;        \int_X\mu(\diff{x}) \: \varphi(x) \: \partial_hf(x)       \;=\;       (-1)^{|h|} \int_X  \partial_h[\mu \varphi](\diff{x}) \: f(x)      \;=\;     \pairing{\mathcal{T}_1 \gamma}{f}$$
	To prove the second claim we take $F \in \TM'_q(\W), \: \varphi \in C_B^\infty\big[ X, L(\V) \big]$ and use the Leibniz formula \ref{48} together with the representation $\partial_{\alpha} \mu = \mu \cdot \Theta_{\alpha}$ to obtain:
	\begin{align*}
		\mathcal{T}'_1(F)[\mu\varphi]     \;&:=\;      F \circ \mathcal{T}_1[\mu\varphi]      \;=\;      F\big[ (-1)^{|h|} \partial_h[\mu\varphi] \big]     \;=\;    (-1)^{|h|} \: F\left[  \sum_{\alpha+\beta=h} C(\alpha,\beta) \cdot \mu \cdot \Theta_{\alpha} \cdot \partial_{\beta}\varphi  \right]   \\
		&=\;     (-1)^{|h|}  \sum_{\alpha+\beta=h} C(\alpha,\beta) \: F\big[  \mu \cdot \Theta_{\alpha} \cdot \partial_{\beta}\varphi \big]    \;=\;      (-1)^{|h|}  \sum_{\alpha+\beta=h} C(\alpha,\beta)  \pairing{\mu \cdot \Theta_{\alpha} \cdot \partial_{\beta}\varphi}{F}
	\end{align*}
\end{proof}

Finally, the following definition is now fully motivated.

\begin{definition}[Generalized derivative along a multivector]
	Suppose $h \in M_0^kX$ is a multivector of order not exceeding $k$.\\
	The operation of generalized differentiation along $h$ is defined as follows:
	$$D_h   \::\:  \TM'_q(\W)    \:\ni\:     F    \:\mapsto\:    \TM'_B(\W)$$
	$$\pairing{\mu \varphi}{D_hF}       \;:=\;     (-1)^{|h|}  \sum_{\alpha+\beta=h} C(\alpha,\beta)  \pairing{\mu \cdot \Theta_{\alpha} \cdot \partial_{\beta}\varphi}{F}$$
\end{definition}

For a test function $f \in \TF$ this definition coincides with the integration by parts formula \ref{45}.
Therefore, as intended by the general scheme \ref{54}, the generalized derivative is consistent with the classical one in the sense of its action on the space $\TM_B$.


\newpage
\section{Dobrakov-Sobolev spaces} \label{58}

In this section we continue to work in the environment and notation set up in the preceding section.
Due to the results established therein, for any element of $L_p(\mu,\V)$ and any multivector $h$ of order not exceeding $k$ there is a generalized derivative defined in accordance with the scheme:
$$L_p(\mu,\V)    \;\overset{\mathcal{R}_q}{\hookrightarrow}\;    \TM'_q(\W)   \;\xrightarrow{\; D_h \;}\;  \TM'_B(\W)$$
In analogy with the classical definition of Sobolev spaces, we can now consider those functions from $L_p(\mu,\V)$, for which all the generalized derivatives of order up to and including $k$ still remain in the same class.
Thus, we are naturally led to the following definition of a Dobrakov-Sobolev space of order $(k,p)$:

\begin{definition}[Dobrakov-Sobolev space] \label{55}
	$$\DS_p^k(\mu,\V)   \;:= \;  \Big\{ f \in L_p(\mu,\V)  \;\Big\vert\;  \forall h \in M_0^kX \::\: D_h \big[ \mathcal{R}_q f \big]  \in  \mathcal{R}_B \big[ L_p(\mu,\V) \big]\Big\}$$
	Dropping the explicit mention of the injections $\mathcal{R}_q$ and $\mathcal{R}_B$, we shall write simply $D_h f \in L_p(\mu,\V)$.\\
	The corresponding Sobolev norm is defined with the help of Dobrakov semivariation:
	$$\big\| f \big\|_p^k   \;:=\;    \sum_{0 \leq |h| \leq k} \big\| D_h f \big\|_p    \:=\:    \sum_{0 \leq |h| \leq k} \widehat{\mu}_p\big(D_h f , X\big)$$
\end{definition}

Building on the results obtained previously, it is possible to give the following equivalent description of this class:
\begin{proposition} \label{56}
	A function $f \in L_p(\mu,\V)$ belongs to the class $\DS_p^k(\mu,\V)$, if for any multivector $h \in M_0^kX$ there exists a function $g_h \in L_p(\mu,\V)$ satisfying:
	$$ \forall \varphi \in C_B^\infty\big[ X, L(\V)\big] :\; \int_{X} \mu(\diff x) \: \varphi(x) \: g_h(x)        \;\:=\;\:       (-1)^{|h|}  \sum_{\alpha+\beta=h} C(\alpha,\beta) \int_X \mu(\diff{x}) \: \Theta_{\alpha}(x) \: \partial_{\beta}\varphi(x) \: f(x)$$
\end{proposition}

Comparing Definition \ref{0} with Proposition \ref{56}, first, we observe the natural transition from scalar-valued to operator-valued multipliers.
Second, if Proposition \ref{56} is postulated as the definition by analogy with \ref{0}, then the notion of Sobolev derivative $g_h$ precisely matches that of generalized derivative, obtained in the previous section.
To conclude, we show that our space is complete.

\begin{theorem}
	$\Big( \DS_p^k(\mu,\V) , \|\cdot\|_p^k \Big)$ is a Banach space over $\F$.
\end{theorem}
\begin{proof}
	Let $\{f_n\}_1^\infty$ be a Cauchy sequence in $\DS_p^k(\mu,\V)$.
	By definition of the Sobolev norm this means that for any multivector $h \in M_0^kX$ the sequence $\{D_h f_n\}_{n=1}^{\infty}$ is Cauchy in $L_p(\mu,\V)$.
	By completeness of the space $L_p(\mu,\V)$:
	$$\exists f\in L_p(\mu,\V) :\: f_n \xrightarrow{L_p} f    \qquad\text{and}\qquad \forall 1 \leq |h| \leq k: \: \exists g_h \in L_p(\mu,\V): \: D_h f_n \xrightarrow{L_p} g_h$$
	It remains only to show that for any multivector the distributional equality $D_h f = g_h$ holds.\\
	For all $\varphi \in C_B^\infty\big[X, L(\V)\big]$ and $\alpha+\beta=h$ H\"older's inequality \ref{23} gives:
	\begin{multline*}
	$$\big\| \Theta_{\alpha} \cdot \partial_{\beta}\varphi \cdot f_n    -    \Theta_{\alpha} \cdot \partial_{\beta}\varphi \cdot f \big\|_1         \;=\;         \widehat{\mu}_1\big[ \Theta_{\alpha} \cdot \partial_{\beta}\varphi \cdot (f_n - f) \:,\: X\big]       \;\leq \\
	\leq\;\:        \widehat{\mu}_q\big[ \Theta_{\alpha} \:,\: X \big] \cdot \widehat{\mu}_p\big[ \partial_{\beta}\varphi \cdot (f_n - f) \:,\: X \big]         \:\;\leq\;\:       \big\|\Theta_\alpha\big\|_q \cdot \big\|\partial_{\beta}\varphi\big\|_X \cdot \big\|f_n - f\big\|_p   \;\xrightarrow{n\to\infty}\;    0$$
	\end{multline*}
	\begin{multline*}
	$$\big\| \Theta_{\alpha} \cdot \partial_{\beta}\varphi \cdot D_hf_n    -    \Theta_{\alpha} \cdot \partial_{\beta}\varphi \cdot g_h \big\|_1         \;=\;         \widehat{\mu}_1\big[ \Theta_{\alpha} \cdot \partial_{\beta}\varphi \cdot (D_h f_n - g_h) \:,\: X\big]       \;\leq \\
	\leq\;\:        \widehat{\mu}_q\big[ \Theta_{\alpha} \:,\: X \big] \cdot \widehat{\mu}_p\big[ \partial_{\beta}\varphi \cdot (D_h f_n - g_h) \:,\: X \big]         \:\;\leq\;\:       \big\|\Theta_\alpha\big\|_q \cdot \big\|\partial_{\beta}\varphi\big\|_X \cdot \big\|D_h f_n - g_h\big\|_p   \;\xrightarrow{n\to\infty}\;    0$$
	\end{multline*}
	Hence we have: 
	$$ \Theta_{\alpha} \cdot \partial_{\beta}\varphi \cdot f_n     \;\xrightarrow{L_1}\;  \Theta_{\alpha} \cdot \partial_{\beta}\varphi \cdot f        \qquad\text{and}\qquad           \Theta_{\alpha} \cdot \partial_{\beta}\varphi \cdot D_h f_n     \;\xrightarrow{L_1}\;  \Theta_{\alpha} \cdot \partial_{\beta}\varphi \cdot g_h$$
	Now we verify the desired equality pointwise on every test measure from the space $\TM_B$:
	\begin{align*}
		\pairing{\mu \varphi}{D_h f}     \;&=\;       (-1)^{|h|}  \sum_{\alpha+\beta=h} C(\alpha,\beta)  \pairing{\mu \cdot \Theta_{\alpha} \cdot \partial_{\beta}\varphi}{f} \\
		&=\;      (-1)^{|h|}  \sum_{\alpha+\beta=h} C(\alpha,\beta)  \int_X \mu(\diff{x}) \: \Theta_{\alpha}(x) \: \partial_{\beta}\varphi(x) \: f(x)    \\
		&=\;      (-1)^{|h|}  \sum_{\alpha+\beta=h} C(\alpha,\beta) \: \lim_{n\to\infty}  \int_X \mu(\diff{x}) \: \Theta_{\alpha}(x) \: \partial_{\beta}\varphi(x) \: f_n(x) \\
		&=\;      \lim_{n\to\infty} \pairing{\mu \varphi}{D_h f_n}      \;=\;     \lim_{n\to\infty} \int_X \mu(\diff{x}) \: \varphi(x) \: D_h f_n(x)   \\
		&=\;      \int_X \mu(\diff{x}) \: \varphi(x) \: g_h(x)        \;=\;       \pairing{\mu \varphi}{g_h}
	\end{align*}
	Thus, by definition $f \in \DS_p^k(\mu,\V)$ and $f_n \to f$ in the Sobolev norm:
	$$\big\| f_n - f \big\|_p^k    \:=\:   \big\| f_n - f \big\|_p     \:+\:      \sum_{1 \leq |h| \leq k} \big\| D_h f_n - g_h \big\|_p   \:\xrightarrow{n\to\infty}\:  0$$
\end{proof}

\vspace{2cm}

\begin{note}
	\noindent
	This paper is based on the author's Master's thesis defended in 2026 at Lomonosov Moscow State University, Faculty of Mechanics and Mathematics (thesis supervisor: Professor Igor Anatolievich Sheipak).
\end{note}

\newpage
\printbibliography

\end{document}